\documentclass[preprint,10pt]{elsarticle}
\usepackage{amsmath,amssymb,amsthm,mathrsfs,enumerate}
\newcommand{\ch}{\theta_0}
\newcommand{\cht}{\theta}
\newcommand{\pir}{\pi}

\newcommand{\Res}{\mathrm{Res}}
\newcommand{\Ind}{\mathrm{Ind}}
\newcommand{\dcrep}{{\mathsf{d}}}
\newcommand{\ecrep}{{\mathsf{e}}}

\newcommand{\Mu}{M}

\newcommand{\R}{\mathcal{R}}
\newcommand{\U}{\mathcal{U}}
\newcommand{\Z}{\mathbb{Z}}
\newcommand{\SL}{\mathrm{SL}}

\newcommand{\PGL}{\mathrm{PGL}}
\newcommand{\GL}{\mathrm{GL}}
\newcommand{\Gal}{\mathrm{Gal}}

\newcommand{\K}{\mathcal{K}}
\newcommand{\G}{\mathcal{G}}
\newcommand{\GG}{\mathbb{G}}

\newcommand{\lrc}[1]{\lceil #1 \rceil}
\newcommand{\g}{\mathfrak{g}}
\newcommand{\T}{\mathcal{T}}
\newcommand{\TT}{\mathbb{T}}
\newcommand{\Stor}{\mathbb{S}}
\newcommand{\ratk}{k}
\newcommand{\resk}{\kappa}
\newcommand{\extk}{k'}
\newcommand{\extresk}{\kappa'}

\newcommand{\LieT}{\mathfrak{t}}
\newcommand{\ep}{\varepsilon}
\newcommand{\PP}{\mathcal{P}}
\newcommand{\p}{\varpi}
\newcommand{\val}{\mathrm{val}}
\newcommand{\real}{\mathbb{R}}
\newcommand{\Rplus}{\widetilde{\mathbb{R}}}
\newcommand{\A}{\mathcal{A}}
\newcommand{\B}{\mathcal{B}}
\newcommand{\Aphi}{\Gamma}
\newcommand{\aphi}{\mathsf{a}}

\newcommand{\alp}{\alpha}
\newcommand{\I}{\mathcal{I}}
\newcommand{\cind}{\textrm{c-}\mathrm{Ind}}
\newcommand{\sig}{\sigma}
\newcommand{\om}{\omega}

\newcommand{\gam}{\gamma}
\newcommand{\Tr}{\mathrm{Tr}}
\newcommand{\Sh}{\mathcal{S}}

\newcommand{\ZZ}{Z}

\newcommand{\sgn}{\mathop{sgn}}
\newcommand{\w}{w}
\newcommand{\diag}{\textrm{diag}}

\newcommand{\Gmess}[1][d/2]{\G_{[0,\frac12],#1}}

\newcommand{\mat}[1]{\left[ \begin{matrix} #1 \end{matrix} \right]}
\newcommand{\smat}[1]{\left[ \begin{smallmatrix} #1 \end{smallmatrix} \right]}

\theoremstyle{plain}
\newtheorem{theorem}{Theorem}[section]

\newtheorem{lemma}[theorem]{Lemma}
\newtheorem{proposition}[theorem]{Proposition}
\newtheorem{corollary}[theorem]{Corollary}

\theoremstyle{definition}

\newtheorem{remark}[theorem]{Remark}

\theoremstyle{remark}

\journal{Algebra}

\begin{document}

\begin{frontmatter}
\title{Branching Rules for Supercuspidal Representations of $\SL_2(\ratk)$, for $\ratk$ a $p$-adic field}
\author{Monica Nevins\fnref{label2}}
\fntext[label2]{This research is supported by a Discovery Grant from NSERC Canada.}
\address{Department of Mathematics and Statistics, University of Ottawa, Ottawa, Canada K1N 6N5}
\ead{mnevins@uottawa.ca}
%\thanks{This research is supported by a Discovery Grant from NSERC Canada.}
%\subjclass{20G05}

\date{\today}
\begin{abstract}
The restriction of a supercuspidal representation of $\SL_2(\ratk)$ to a maximal compact subgroup decomposes as a multiplicity-free direct sum of irreducible representations.  We explicitly describe this decomposition, and determine how the spectrum of this decomposition varies as a function of the parameters describing the supercuspidal representation.  
%This work builds on work of K.~Hansen for the case $\mathrm{GL}_2(\ratk)$, but uses instead the building-theoretic construction of supercuspidal representations given by J.~K.~Yu, applying it to the special and instructive case of $\SL_2(\ratk)$.
\end{abstract}
\end{frontmatter}
%\maketitle

\section{Introduction}

The study of branching rules is that of considering the decomposition of an irreducible representation to an interesting subgroup, with the goal of revealing additional internal symmetries of the original representation, and exposing commonalities in families of representations.  In the case of Lie groups, the most interesting subgroup to consider is a maximal compact subgroup (which is unique up to conjugacy) and the resulting theory of minimal $K$-types \cite{Vogan1979} is a significant milestone in that representation theory.

When we turn to the case of a $p$-adic reductive algebraic group $\G$, there are several immediate differences.  Maximal compact subgroups are open, and not themselves $p$-adic algebraic groups; moreover, there exist in general several conjugacy classes of maximal compact subgroups.  Nevertheless, through the work of A.~Moy, G.~Prasad, C.~Bushnell, P.~Kutzko and others \cite{MoyPrasad1994,BushnellKutzko1993}, a kind of analogue of the theory of minimal $K$-types has emerged, now called the \emph{theory of types}.  The various compact subgroups arising in this theory are open but generally not maximal.  

It remains an open question to characterize the representations of a given maximal compact subgroup $\K$ which occur in the decomposition of an irreducible representation of $\G$.  Aspects of this question were considered for the group $\PGL_2(\ratk)$ by A.~Silberger \cite{Silberger1970, Silberger1977} and for $\GL_2(\ratk)$ by W.~Casselman \cite{Casselman1973} and K.~Hansen \cite{Hansen1987}.  Branching rules for the Weil representation of $\mathrm{Sp}_{2n}(\ratk)$ were considered for the restriction to one maximal compact subgroup by D.~Prasad in \cite{Prasad1998} and a more explicit decomposition, relative to all conjugacy classes of maximal compact subgroups, was given by K.~Maktouf and P.~Torasso in \cite{MaktoufTorasso2011}.  The author considered the case of principal series of $\SL_2(\ratk)$ in \cite{Nevins2005} and, together with P.~Campbell, addressed principal series of $\GL_3(\ratk)$ in \cite{CampbellNevins2009, CampbellNevins2010}.  Recently, U.~Onn and P.~Singla  completed the study of branching rules of unramified principal series of $\GL_3(\ratk)$ by giving an explicit description of the decomposition into irreducible representations of $\K$ \cite{OnnSingla2012}.
The work with $\GL_3(\ratk)$, along with related calculations by U.~Onn, A.~Prasad and L.~Vaserstein on sizes of double coset spaces in \cite{OnnPrasadVaserstein2006},  makes it clear that the question of decomposing principal series in the general case will be highly nontrivial.  

The work to date suggests that the presentation of a satisfactory answer is tied to the development of the representation theory of Lie groups over finite local rings.  This theory is of high current interest but as yet far from complete; as only one example note the work of A.-M.~Aubert, U.~Onn, A.~Prasad and A.~Stansinski \cite{AubertOnnPrasad2010}.

We would also like to signal the closely related work of J.~Lansky and A.~Raghuram \cite{LanskyRaghuram2007}, whose focus is the determination of newforms for $\SL_2(\ratk)$; this amounts to identifying and characterizing the irreducible component of least depth (and least degree) occuring in the restriction of a representation of $\SL_2(\ratk)$ to a maximal compact subgroup.

Constructions of supercuspidal representations, phrased in terms of compact open subgroups which arise from the Bruhat-Tits building of $\G$, offer the possibility of a geometric, or at least building-theoretic, approach to the description of the branching rules of the corresponding representations.  These include works by J.~Adler~\cite{Adler1998}, L.~Morris \cite{Morris1991,Morris1992} and more recently J.~K.~Yu \cite{Yu2001}, whose construction of tame supercuspidal representations was shown to be exhaustive by J.~Kim \cite{Kim2007}.

  The present work is the first step of a longer term project of describing branching rules for tame supercuspidal representations.  
In this paper, we consider the restriction of  supercuspidal representations of $\SL_2(\ratk)$, where $\ratk$ is a local nonarchimedean field of residual characteristic different from $2$, to the maximal compact subgroup $\K = \SL_2(\R)$, where $\R$ denotes the integer ring of $\ratk$.  The main results on branching rules are given in Theorems~\ref{T:depthzero} and \ref{T:positivedepth}.  An explicit description of all irreducible representations of $\K$ was given by J.~Shalika~\cite{Shalika1967}, so our answer is given as a sum of known representations.
 We then analyse the intertwining over $\K$ of different supercuspidal representations of $\G$, culminating in Corollary~\ref{C:intertwining} and Theorem~\ref{T:intertwining}.

This work, together with \cite{Nevins2005}, completes the branching rules for $\SL_2(\ratk)$.  Some consequences of these results, based on an early version of the present paper, are given in \cite{Nevins2011}.

The methods and approach used here are greatly inspired by those in \cite{Hansen1987}.  
%In particular, the alternate description of the irreducible components presented in Proposition~\ref{P:alternate} owes its existence to the analogous result \cite[Theorem?]{Hansen1987} proven by K.~Hansen.  
For instance, the idea of using character calculations in Section~\ref{S:depthzero} to resolve the branching rules in the depth-zero case comes from the similar approach of K.~Hansen.  In fact, although our presentation makes no use of the corresponding results for $\GL_2(\ratk)$, the description of the decomposition could be accomplished via the following alternate route, which was used in the author's first draft.  Each tame supercuspidal representation of $\SL_2(\ratk)$ appears in the restriction to $\SL_2(\ratk)$ of some supercuspidal representation of $\GL_2(\ratk)$ \cite{MoySally1984}.  The decomposition of the restriction to $\GL_2(\R)$ of a supercuspidal representation of $\GL_2(\ratk)$ is given in \cite{Hansen1987}.  Consequently, it suffices to describe the restriction of each of these representations of $\GL_2(\R)$ to $\SL_2(\R)$ and identify how the various pieces are apportioned between the supercuspidal representations of $\SL_2(\ratk)$.  

Conversely, from the present work and the literature one may deduce that, given 
a supercuspidal representation of $\GL_2(\ratk)$  of depth $r$, 
all of the irreducible representations of $\GL_2(\R)$ of depth greater than $r$ which occur in its restriction decompose, upon further restriction, to a direct sum of two irreducible representations of $\SL_2(\R)$.  Furthermore, the methods of Section~\ref{S:intertwining} could easily be used to extend the work of K.~Hansen to determine the intertwining over $\GL_2(\R)$ of different supercuspidal representations.  Note that the $\GL_2$ analogue of Corollary~\ref{C:intertwining} was previously known \cite{Casselman1973} and is much simpler.

The advantage of the current approach to the branching rules for $\SL_2(\ratk)$ is two-fold.  For one, the present classification of supercuspidal representations of $\SL_2(\ratk)$ is tight: work by J.~Hakim and F.~Murnaghan \cite{HakimMurnaghan2008} has answered the question of equivalence of supercuspidal representations in J.~K.~Yu's construction as a function of the tamely ramified cuspidal $\G$-data of their construction; whereas this is less clear from the description of supercuspidal representations of $\GL_2(\ratk)$ used in \cite{Hansen1987}.  For another, this building-theoretic approach seems to be a most promising language for considering tame supercuspidal representations of general connected reductive $p$-adic groups.

In this paper we consider only one conjugacy class of maximal compact subgroup, namely $\K = \SL_2(\R)$.  There is another conjugacy class, represented by $\K^\eta$, where $\eta\in \GL_2(\ratk)$ is as in \eqref{E:eta}.  As was shown explicitly in \cite{Nevins2005} for the principal series of $\SL_2(\ratk)$, the conjugacy of these two groups under $\GL_2(\ratk)$ implies that their branching rules are completely analogous.

Generalizing the present work to higher rank cases poses significant challenges since, for example, in our arguments we do at several points use explicit descriptions of the anisotropic tori in $\G$.  Also, to parametrize certain double cosets, and in some calculations, we represent elements of the group in matrix form.  Some of these difficulties are surmountable, using for example the work of S.~Debacker~\cite{DeBacker2006} to describe the tori, and accepting a less strongly explicit description of the representations which arise.  That said, it is expected that any success towards the general case would advance the representation theory of the maximal compact subgroups, and so is intrinsically interesting.

The paper is organized as follows.  In Section~\ref{S:Notation}, we establish our notation.  We also describe some Moy-Prasad filtration subgroups and give a list of representatives of the anisotropic tori in $\SL_2(\ratk)$.  In Section~\ref{S:supercuspidals}, we give the known parametrization of all supercuspidal representations of $\SL_2(\ratk)$, which is naturally divided into two cases:  zero depth and positive depth.  We also present a subset of J.~Shalika's parametrization of representations of $\SL_2(\R)$.   We describe the necessary tools for restricting representations in Section~\ref{S:restriction}, including a statement of Mackey theory applicable to the present setting, and then we present a first decomposition of the supercuspidal representations upon restriction to $\SL_2(\R)$.  

Section~\ref{S:depthzero} is devoted to the depth-zero case.  We decompose into irreducible representations the components obtained through the Mackey decomposition, and identify their simple uniform description relative to Shalika's parametrization.  For the case of positive depth the corresponding results are proven in Section~\ref{S:posdepth}.  We conclude in Section~\ref{S:intertwining} with an analysis of the intertwining over $\SL_2(\R)$ of distinct supercuspidal representations of $\SL_2(\ratk)$.
%, as well as discussing the equivalence of these results for the case of restriction to another (non-conjugate) maximal subgroup of $\SL_2(\ratk)$.

%??Will we say anything about GL2?

\section{Notation and Background} \label{S:Notation}

\subsection{General notational conventions} \label{SS:general}
Let $\ratk$ be a local non\-arch\-i\-me\-de\-an field of residual characteristic $p\neq 2$.  Its characteristic may be $0$ or $p$.  The results of this paper are all valid in this general setting, but for the sake of brevity we will refer to our field as a $p$-adic field and our group as a $p$-adic group.

Denote the residue field of $\ratk$ by $\resk$, a finite field of order $q$.  Let the integer ring of $\ratk$ be $\R$ and its maximal ideal $\PP$.  Let $\p$ be a uniformizer, and normalize the valuation on $\ratk$ so that $\val(\p)=1$.  Then $\ratk^\times/(\ratk^\times)^2$ can be represented by $\{1, \ep, \p, \ep \p\}$ where $\ep$ is a fixed nonsquare in $\R^\times$ (which is chosen to be $-1$ when $-1 \notin (\ratk^\times)^2$).  We shall use $\extk = \ratk[\gamma]$ to denote a quadratic extension field of $\ratk$ by an element $\gamma$ such that $\gamma^2 \in \{\ep, \p, \ep \p\}$; these give all quadratic extensions of $\ratk$ (up to isomorphism).  The units of $\R$ admit a filtration by subgroups $\U_n$ where we set
$$
\U_n = \begin{cases} \R^\times & \textrm{if $n=0$,}\\ 1+\PP^n & \textrm{if $n>0$.}
\end{cases}
$$
We fix an additive quasi-character $\Psi$ of $\ratk$ which is trivial on $\PP$ but nontrivial on $\R$.

Given a subgroup $H$ of a group $G$ we write $H^g$ for the group $gHg^{-1}$.  If $\sigma$ is a representation of $H$ we write $\sigma^g$ for the corresponding representation of $H^g$ given on elements $h$  of $H^g$ by $\sigma^g(h) = \sigma(g^{-1}hg)$.  Thus unfortunately we have $\sigma^{gh} = (\sigma^h)^g$.

Given a closed subgroup $K$ of a connected reductive $p$-adic group $G$, and a representation $(\pi,V)$ of $K$, the compactly induced representation $\cind_K^G \pi$ is given by right action by $G$ on the space of functions
$$
\left\{f \colon G \to V \left| \textrm{\parbox{3.5in}{$\forall k\in K, g\in G, f(kg) = \pi(k)f(g)$, $f$ is smooth and is compactly supported mod $K$}}\right.\right\}.
$$
By Mautner's theorem, if $K$ is open and compact mod the center $Z$ of $G$, and if $\cind_K^G \pi$ is irreducible, then it is supercuspidal.  It is a lasting conjecture, proven now in many cases \cite{Yu2001,Kim2007,Stevens2008} that all supercuspidal representations of $G$ arise in this way, for some choice of $K$ and $\pi$.

Now let $\GG = \SL_2$ denote the algebraic group of $2\times 2$ matrices of determinant one and set $\G = \GG(\ratk)$.  Its center $Z$ is the two-element group $\{\pm I\}$.  Let $\g$ denote its Lie algebra and $\g^*$ its algebraic dual.  Throughout, we abuse notation by writing $\g=\g(\ratk)$ and $\g^* = \g^*(\ratk)$.  Given subsets $S_i$ of $\ratk$ or $\ratk^\times$, we define a corresponding set of matrices with the notation
$$
\left\{ \mat{S_1&S_2\\S_3&S_4} \right\} = \left\{ \left. \mat{a_1&a_2\\a_3&a_4}\right| a_i \in S_i\right\}.
$$
We also write $\diag(a,b)$ for a diagonal matrix with diagonal entries $a,b \in \ratk$, and will write $X(u,v)$ for an antidiagonal matrix starting in Section~\ref{SS:repR}.

For any $Y \in \g^*$, the map $X \mapsto \Psi(\langle Y,X\rangle)$ defines a smooth quasi-character of $\g$, and all smooth quasi-characters of $\g$  arise in this way.  In the presence of a nondegenerate bilinear form on $\g$ (taken without loss of generality to be the trace form $\Tr$ in this case) we can and do simply parametrize these quasi-characters by elements of $\g$.  Thus for $Y \in \g$ we have a quasi-character $X \mapsto \Psi(\Tr(YX))$.
%we write 
%\begin{equation} \label{E:PsiY}
%\Psi^{\g}_Y(X) = \Psi(\Tr(XY)).
%\end{equation}
While convenient for us, this shortcut is not used in \cite{Yu2001}, where the use of the dual Lie algebra throughout permits a uniform general construction.

\subsection{Filtration subgroups in $\SL_2(\ratk)$}

Let $\Stor$ denote a maximal torus of $\GG$, split over $\ratk$, with associated root system $\Phi$.  Let $\A = \A(\GG, \Stor, \ratk)$ denote the corresponding apartment (as in \cite[1.2]{Tits1979}).  We may think of $\A$ as the affine space under $X_\ast(\Stor) \otimes \mathbb{R}$, where $X_\ast(\Stor)$ is the group of $\ratk$-rational cocharacters of $\Stor$.  Let $\B = \B(\GG,\ratk)$ denote the reduced Bruhat-Tits building for $\GG$ over $\ratk$ and let $y\in \B$.   When, as here, $\G$ is semisimple, simply connected and split over $\ratk$, the stabilizer in $\G$ of $y$ is equal \cite[\S 3.1]{Tits1979} to $\G_{y,0}$, the connected parahoric subgroup associated with $y$, as described by A.~Moy and G.~Prasad in \cite{MoyPrasad1994}.  It is thus unambiguous to write $\G_y$ for $\G_{y,0}$.

Define $\Rplus = \real \cup (\real+) \cup \{\infty\} $ as in \cite[6.4.1]{BruhatTits1972} and associate to each $y \in \A, a\in \Phi$ and $r \in \Rplus$ a subgroup $\GG_a(\ratk)_{y,r}$ of the corresponding root subgroup as well as the  associated $\R$-submodule $\g_{a,y,r} \subset \g$.  For $r \geq 0$, one can similarly define filtration subgroups on the torus, $\Stor(\ratk)_{y,r}$.  Then the Moy-Prasad filtration subgroup $\G_{y,r}$ is the group generated by $\{\Stor(\ratk)_{y,r}, \GG_a(\ratk)_{y,r} \mid a\in \Phi\}$.  Given a subset $I\subseteq \A$, one sets $\G_{I,r} = \cap_{y \in I}\G_{y,r}$.  The lattices $\g_{y,r}$ in $\g$, for $y\in \A$ and $r\in \Rplus$, are defined similarly.

More concretely, for the case $\G=\SL_2(\ratk)$ we choose $\Stor$ to be the diagonal torus, with $\Phi = \{\pm \alpha\}$.  Then the root subgroup $\GG_{\alpha}$ (respectively $\GG_{-\alpha}$) is the group of unit upper triangular (respectively lower triangular) matrices in $\GG$.  By identifying an origin $y=0$ in $\A$, and choosing coordinates so that $\alpha^\vee = 2$, the vertices of the simplicial complex representing the action of the affine Weyl group on $\A$ occur for integer coordinates.  We have
%$\alpha(y) = y$ for each $y\in\Z$, 
$\G_{0} = \SL_2(\R)$; this is a maximal compact subgroup of $\G$ which we denote by $\K$ throughout.  On the other hand, at $y=1$ we obtain  
\begin{equation} \label{E:Ktilde}
\G_{1} = \left\{ \mat{\R & \PP^{-1} \\ \PP & \R}\right\}  \cap \G = \K^\eta
\end{equation}
where $\eta \in \GL_2(\ratk)$ is given by
\begin{equation} \label{E:eta}
\eta = \mat{1 & 0\\0 & \p}.
\end{equation}
Then $\K$ and $\K^\eta$ are representatives of the two conjugacy classes of maximal compact subgroups of $\G$.  If we define, for finite $r \in \Rplus$, 
$$
\lceil r \rceil = \begin{cases} 
\min\{n \in \Z \mid n \geq r\} & \textrm{if $r\in \real$, and}\\
\min\{n \in \Z \mid n > r\} & \textrm{if $r\in \real+$,}
\end{cases}
$$
then for any $y\in \A$ and $r\geq 0$ we have
$$
\GG_{\alpha}(\ratk)_{y,r} = \left\{ \mat{1& \PP^{\lrc{r-y}} \\ 0 & 1} \right\}, \quad \GG_{-\alpha}(\ratk)_{y,r}  = \left\{ \mat{1& 0 \\ \PP^{\lrc{r+y}} & 1} \right\}
$$
and $\Stor(\ratk)_{y,r}$ consists of those diagonal matrices in $\G$ with entries in $\U_{\lrc{r}}$.
% $\GG_{\alpha}(\ratk)_{y,r}$ as the group of unit upper triangular matrices $X$ such that the nonzero entry of $X-I$ lies in $\PP^{\lceil r-y \rceil}$, $\GG_{-\alpha}(\ratk)_{y,r}$ as the unit lower triangular matrices $X$ such that the nonzero entry of $X-I$ lies in $\PP^{\lceil r+y \rceil}$, and $\Stor(\ratk)_{y,r}$ as those diagonal matrices in $\G$ with entries in $\U_{\lrc{r}}$.  %$\R^\times$ (if $r = 0$) or in $1+\PP^{\lceil r \rceil}$ (otherwise).  
Thus for finite positive $r\in \Rplus$ we have
\begin{equation} \label{E:Gyr}
\G_{y,r} = \left\{ \mat{\U_{\lceil r \rceil} & \PP^{\lceil r-y \rceil} \\ \PP^{\lceil r+y \rceil} & \U_{\lceil r \rceil}} \right\} \cap \G.
\end{equation}
Note that the Moy-Prasad filtration subgroups $\G_{0,r}$ of $\K$ and $\G_{1,r}$ of $\K^\eta$, for $r \geq 0$, are just the standard congruence subgroups $\K_r$ and $\K^\eta_r$, with distinct ones indexed by $r\in \mathbb{N}$.

Given an irreducible smooth representation $(\pi,V)$ of $\G$, the \emph{depth} of $\pi$ is defined as the least $r\in \real_{ \geq 0}$ such that there exists $x\in \B(\GG,\ratk)$ for which $V$ contains vectors invariant under $\G_{x,r+}$.  We also refer to the depth of a representation of $\G_x$, for fixed $x$.

\subsection{Anisotropic tori in $\G$}

The construction of supercuspidal representations of $\G$ proceeds from its non-split tori.  In $\SL_2(\ratk)$ all non-split tori are totally anisotropic and split over a quadratic extension $\extk$ of $\ratk$. % as in Section~\ref{SS:general}. 
A torus is called \emph{unramified} if $\extk$ is unramified, and \emph{ramified} otherwise.  We give a list of representatives of conjugacy classes of anisotropic tori in Table~\ref{Table:tori}; these are well-known.  There are six when $-1 \in (\ratk^\times)^2$; otherwise, we have $\T_{1,\p} \cong \T_{-1,-\p}$ and $\T_{1,-\p} \cong \T_{-1,\p}$ and there are only four conjugacy classes.  

The last column of Table~\ref{Table:tori} is needed for specifying, among other things, the Moy-Prasad filtration on the torus and its Lie algebra, and is determined as follows.

Suppose $\TT$ is an anisotropic torus of $\GG$ over $\ratk$.  Since $p \neq 2$, the extension $\extk$ over which it splits is tamely ramified and so by \cite[2.6.1]{Tits1979} one can view $\B$ as the subset of $\B(\GG,\extk)$ fixed by $\Gal(\extk/\ratk)$
and identify $\A(\GG,\TT,\ratk) = \A(\GG,\TT,\extk) \cap \B(\GG,\ratk)$.   Since $\TT(\ratk)$ is totally anisotropic, this intersection consists of a single point (which one also sometimes has cause to identify with $\B(\TT,\ratk)$).  

To determine this intersection, we use the identification $\B(\GG,\extk) = (\GG(\extk) \times \A)/\sim$, where $(g,x)\sim (h,y)$ if and only if there exists $n \in N$, the normalizer of $\Stor(\extk)$, so that $n \cdot x  = y$ and $g^{-1}hn \in \GG(\extk)_x$ \cite[(7.4.1)]{BruhatTits1972}.  Since $\TT$ splits over $\extk$, there is a $g\in \GG(\extk)$ such that $\TT = \Stor^g$, and $\A(\GG,\TT,\extk)=g\cdot \A$ is identified with the image of $\{g\} \times \A$ in $\B(\GG,\extk)$.  Write $[g,x]$ for the point in the building corresponding to the equivalence class of the pair $(g,x)$; then for example $[1,x]=x$ in our previous notation.
The Galois group action on $\B(\GG,\extk)$ is given by $\sigma([g,x]) = [\sigma(g),x]$ for each $\sigma \in \Gal(\extk/\ratk)$.  It follows that the Galois-fixed points of $\B(\GG,\extk)$, that is, the elements of $\B(\GG,\ratk)$, are those $[g,x]$ for which $\sigma(g^{-1})g \in \GG(\extk)_x$ for all $\sigma \in \Gal(\extk/\ratk)$.

In our case, the Galois group has order two and we write $\Gal(\extk/\ratk) = \{1,\sigma\}$.  Given $\TT$ such that $\TT(k) = \T_{\gam_1,\gam_2}$ for some $(\gam_1,\gam_2)\in k^2$ as in Table~\ref{Table:tori},  the element
\begin{equation} \label{E:g}
g =\mat{1 & -\frac12 \sqrt{\gam_1\gam_2^{-1}}\\\sqrt{\gam_1^{-1}\gam_2} & \frac12} \in \GG(\extk)
\end{equation}
satisfies $\TT = \Stor^g$.
%Since $\TT$ is totally anisotropic over $\ratk$, there is a unique Galois-fixed point $[g,y]$ of $\A(\GG,\TT,\extk)$.  
For these judicious choices of $(\TT,g)$, we see that both $g$ and $\sigma(g^{-1})$ lie in $\GG(\extk)_{y}$ for the given point $y$.  It follows that the unique Galois-fixed point of $\A(\GG,\TT,\extk)$ is $[g,y]=[1,y]=y \in \A(\GG,\TT,\extk)\cap \A$.

\begin{table}
\begin{center}
\renewcommand{\arraystretch}{1.5}
\begin{tabular}{|llc|c|}
\hline
Anisotropic torus & $t(a,b)=$ &  Splitting field   & $\A(\GG,\TT,\ratk)$\\
$\T_{\gam_1,\gam_2}=\TT(\ratk)$ & $\smat{a&b \gam_1 \\ b\gam_2 & a}$ & $\extk$ & $=\{y\}$ \\ \hline \hline
$\T_{1,\ep}$ & $\smat{a & b\\ b\ep & a}$ & $k(\sqrt{\ep})$ &  $y=0$\\
$\T_{\p^{-1},\ep \p} = \T_{1,\ep}^\eta$ & $\smat{a & b\p^{-1}\\ b\ep\p & a}$ &$k(\sqrt{\ep})$& $y=1$\\
\hline
$\T_{1,\p}$ &  $\smat{a & b\\ b\p & a}$ & $k(\sqrt{\p})$ & $y=\frac12$\\ 
$\T_{\ep,\ep^{-1}\p}$ &  $\smat{a & b\ep\\ b\ep^{-1}\p & a}$ & $k(\sqrt{\p})$ &$y=\frac12$\\ \cline{1-3}
$\T_{1,\ep\p}$ &  $\smat{a & b\\ b\ep\p & a}$ &  $k(\sqrt{\ep\p})$ & $y=\frac12$\\ 
$\T_{\ep,\p}$ &  $\smat{a & b\ep\\ b\p & a}$ &$k(\sqrt{\ep\p})$&$y=\frac12$ \\
\hline
\end{tabular}
\caption{Representatives of equivalence classes of anisotropic tori $\TT(\ratk)$ in $\G$ (with $a,b$ symbols denoting elements in $\R$, and where the torus consists of those matrices $t(a,b)$ of determinant one); $\eta$ was defined in \eqref{E:eta}.  Listed are: defining parameters $(\gam_1,\gam_2)$, the splitting field $\extk$ of the torus and the point $y\in \A$ such that $\{y\} = \A(\GG,\TT,\extk)\cap\B$.   If $-1\notin (\ratk^\times)^2$ then $\T_{1,\p} \cong \T_{\ep,\ep^{-1}\p}$ and $\T_{1,\ep\p} \cong \T_{\ep,\p}$.  } \label{Table:tori}
\end{center}
\end{table}

%Finally, note that by our choices we simply have $\frac12 \val\left(\gam_1^{-1}\gam_2\right) =  y$, which will be used to simplify some expressions.

The filtration on the split torus $\Stor(\extk)$ defines one on $\TT(\extk)$ by conjugation \cite[Section 2.6]{MoyPrasad1994}, and hence on $\T = \TT(\ratk)$ %\T_{\gam_1,\gam_2}$ 
by restriction to the set of Galois-fixed points.   Since $g\in \GG(\extk)_{y}$  and $\T \subseteq \G_y$, this filtration is simply the intersection with $\T$ of the Moy-Prasad filtration of $\G_{y}$.  In particular, $\T_0=\T$ and for each positive $r\in \Rplus$ 
%With $\T = \T_{\gam_1,\gam_2}$ as in Table~\ref{Table:tori}, we thus have for each $r\in \Rplus$, $r > 0$, that 
we have
$$
\T_{r} = \{t(a,b) \in \T \mid a\in \U_{\lceil r \rceil}, b\gam_1 \in \PP^{\lrc{r-y}} \}
$$
when  $\T = \T_{\gam_1,\gam_2}$.
The Lie algebra $\LieT$ of $\T$ is the one-dimensional subalgebra of $\g$ spanned by
$$
X_{\T} = \mat{0 & \gam_1 \\ \gam_2 & 0}.
$$
For any $r \in \Rplus$, the corresponding filtration subring of $\LieT$ is
$$
\LieT_{r} = \{ \aphi X_{\T} \mid \aphi \in \ratk, \; \aphi \gam_1 \in \PP^{\lrc{r-y}} \}.
$$
For any $r \geq 0$, elements $X = \aphi X_\T \in \LieT_{-r}$ satisfying $\val(\aphi \gam_1) = -r-y$ are called \emph{$\G$-generic of depth $r$}.  

\begin{remark}
An element $X \in \LieT_{-r}$ (or more precisely, the element of $\LieT^*$ that it represents via our identification $\LieT \cong \LieT^*$) is $\G$-generic of depth $r$ if it satisfies the two conditions (GE1) and (GE2) of \cite[\S8]{Yu2001}.  Since $p \neq 2$, (GE2) follows from (GE1) by  \cite[Lemma 8.1]{Yu2001}.  Condition (GE1) is that $\val(\Tr(XH_a)) = -r$, where $H_a = d\alpha^\vee(1)$ corresponds to the coroot $\alpha^\vee$ of the unique positive root $\alpha$ of $(\GG,\TT)$.  Explicitly, we have $H_a = \smat{1&0\\0&-1}^{g}$, where $g$ is as defined in \eqref{E:g}, and so we verify directly that (GE1) is equivalent to our stated condition.
\end{remark}

Let $r > 0$.  Then there is a natural group isomorphism $e \colon \LieT_r/\LieT_{r+} \to \T_r/\T_{r+}$.  If $\phi$ is a character of $\T$ of depth $r$, we say that $\phi$ (or $\Res_{\T_r}\phi$) is \emph{$\G$-generic} if there exists a $\G$-generic element $\Aphi \in \LieT_{-r}$ of depth $r$ such that for all $X \in \LieT_{r}$ we have $\phi(e_r(X)) = \Psi(\Tr(\Aphi X))$.  In our case all  positive-depth characters of $\T$ are $\G$-generic.
%$\phi \circ e_r = \Psi^{\g}_{\Aphi}$ (with $\Psi^{\g}_{\Aphi}$ as defined in \eqref{E:PsiY}). %(\Tr(\Aphi Y))$. 

\section{Supercuspidal representations of $\SL_2(\ratk)$} \label{S:supercuspidals}

The representation theory of $\SL_2(\ratk)$ has been known since the 1960s.  In this section we present the classification of supercuspidal representations of $\SL_2(\ratk)$ following the work of J.~K.~Yu \cite{Yu2001}.  This allows us to exploit the modern language of buildings and to describe this elegant theory in the simplest possible case.   Additional sources for the material in this section include the survey paper of J.~Kim \cite{Kim2009} and the exposition in the paper by J.~Hakim and F.~Murnaghan \cite{HakimMurnaghan2008}.  The supercuspidal representations of so-called degree $1$ (which are the only supercuspidal representations of positive depth to occur in $\SL_2(\ratk)$) are also called toral supercuspidal representations; these were first described in this way by J.~Adler in \cite{Adler1998}.

%A clear line divides supercuspidal representations of depth zero from those of positive depth, and these classes will be treated separately here and throughout.

\subsection{Depth-zero supercuspidal representations}
Depth-zero supercuspidal representations are induced from cuspidal representations of  $\GG(\resk)= \SL_2(\resk)$, which are well-known; see, for example, \cite{DigneMichel1991}.  Briefly: let $\extresk$ denote the unique quadratic extension field of $\resk$, and $N \colon \extresk \to \resk$ the norm map.  To each nontrivial character $\om$ of $\ker(N)$, Deligne and Lusztig associate a representation $\sig = \sig(\om)$ of $\SL_2(\resk)$ of degree $q-1$.  When $\om^2 \neq 1$, this representation is cuspidal (a so-called Deligne-Lusztig representation).  When $\om=\om_0$, the nontrivial quadratic character, the representation $\sig_0 \doteq \sig(\om_0)$ decomposes instead as $\sig_0 = \sig_0^+ \oplus \sig_0^-$, a direct sum of inequivalent cuspidal representations of half the degree.  The choice of label $\pm$ is related to a choice of nontrivial additive character on $\resk$; without loss of generality we may assume the inflation to $\R$ of this character coincides with the restriction of $\Psi$ to $\R$, where $\Psi$ was fixed in Section~\ref{SS:general}. %(One often constructs these $\sig(\om)$ as the restrictions to $\SL_2(\resk)$ of cuspidal representations of $\GL_2(\resk)$.)

Let $\sig$ be a cuspidal representation of $\GG(\resk)$.  Inflate $\sig$ to a representation (also denoted $\sig$) of $\GG(\R)=\K$, and let $\sig^\eta$ denote the corresponding representation of $\K^\eta$.  All depth-zero cuspidal representations of $\K$ and $\K^\eta$ arise in this way and it is known that the compactly induced representations
\begin{equation} \label{E:depthzero}
\cind_{\K}^{\G} \sig \quad \textrm{and} \quad  \cind_{\K^\eta}^{\G} \sig^\eta
\end{equation}
are irreducible, hence supercuspidal, and have depth zero.  Note that their central characters coincide with the restriction of the character $\om$ to the $\{\pm 1\}$ subgroup.

As a special case of the general result for depth-zero supercuspidals found in \cite{MoyPrasad1996, Morris1999}, one has the following well-known fact.

\begin{proposition} \label{P:depthzerolist}
Up to equivalence, any depth-zero supercuspidal representation of $\G$ arises as in \eqref{E:depthzero}, for a unique pair $(\sig,\G_{y})$ of cuspidal representation $\sig$ of $\SL_2(\resk)$ and maximal compact subgroup $\G_{y}$ with $y\in \{0,1\}\subset \A$.
\end{proposition}

\subsection{Positive-depth supercuspidal representations}
To construct positive-depth supercuspidal representations, Yu introduced the notion of a \emph{generic tamely ramified cuspidal $G$-datum} \cite[\S 3, \S15]{Yu2001}.  For $\G = \SL_2(\ratk)$, %the \emph{tamely ramified cuspidal $G$-data} 
these data fall into one of  two kinds.  The first consists of data of so-called degree $0$ which give rise as above to the depth-zero supercuspidal representations.  The second, on the subject of which we will devote the rest of this section, consists of data of degree $1$, and each is specified by: a choice of anisotropic torus $\T$, the corresponding point $y\in \A$ as in Table~\ref{Table:tori}, a positive real number $r$,  and a $\G$-generic character $\phi$ of $\T$ of depth $r$.  We note that if $\T$ is unramified then  $r\in\Z$ whereas if $\T$ is ramified then $r \in \frac12 + \Z$.   We set $s= r/2$.

\begin{remark}
Let us relate the abbreviated data given here to the definition of a cuspidal $G$-datum from \cite[\S3,\S15]{Yu2001} or \cite{HakimMurnaghan2008}.  This is given as a sequence of five axioms (D1) to (D5); we further justify statements (D3) to (D5) below.
\begin{description}
\item[D1] The tamely ramified twisted Levi sequence, which is of degree $1$, is $\vec{\GG} = (\GG^0,\GG^1) = (\TT,\GG)$, where the quotient $Z(\TT)/Z(\GG)$ is anisotropic since $\TT$ is;
\item[D2] The point $y$ lies in (is the only point in) $\A(\GG,\TT,\ratk)$;
\item[D3] The sequence of real numbers $0<r_0 \leq r_1$ may be taken to be $0 < r \leq r$;
\item[D4] The irreducible representation $\rho$ of $\T$, such that $\T_{y,0+}$ is $1$-isotypic and $\cind_{\T}^{\T}\rho$ is supercuspidal, may without loss of generality be taken to be trivial;
\item[D5] The sequence of quasi-characters $(\phi_0,\phi_1)$ may be taken to be $(\phi, 1)$ where $\phi$ is a $\G$-generic character of $\T$ of depth $r$.
\end{description}
%Statements D3 and D4 require some justification.  
In (D3), the value $r_0$ defines the depth of the character $\phi_0$ of $\T$ required in (D5), and as such will be either an integer or a half-integer, depending on the type of torus.  One is permitted to chooose some $r_1 > r_0$, but the purpose of the parameter $r_1$ is to allow a twisting by a quasi-character $\phi_1$ of $\G$ of depth $r_1$.  Since $\G$ admits no nontrivial quasi-characters, no such $\phi_1$ exists, so the convention is to set $r_1=r_0$ and take $\phi_1=1$ in (D5).  The resulting representation has depth $r=r_0$, so we exclude $r=0$ here.  

In (D4), any $\rho$ satisfying these conditions must be a depth-zero character of $\T$ (and any depth-zero character would do).  However, for any such $\rho$ one may obtain the same supercuspidal representation in this construction by replacing $\rho$ with the trivial character and replacing $\phi$ by $\rho\phi$. This fact (which can also be seen directly) follows from the equivalence of cuspidal $G$-data given in \cite[Theorem 6.7]{HakimMurnaghan2008}.
\end{remark}

%\begin{remark}
%For the reader of \cite{Yu2001}, we briefly contrast our abbreviated notation here with that which is standard in the discussion of cuspidal $G$-data.    Given $\G^0 \doteq \T$ and $\G^1 \doteq \G$, we have $K^0 \doteq \T$ and $K^1 \doteq \T\G_{y,s}$.  Then $\rho_0' \doteq 1$, $\rho_0 \doteq \phi_0$, and $\rho_1' = \rho_1 \doteq \rho$.   The subgroup $J^1$ is, in the notation of Table~\ref{Table:tori},  the intersection with $\G$ of the subgroup of $\GG(\extk)$ generated by $\TT(\extk)_r$ and $(\GG_{\pm \alpha}(\extk)_{y,s})^{g^{-1}}$.  This subgroup is key in the definition of the Weil representation used to construct $\rho$ when $\G_{y,s+} \neq \G_{y,s}$, but we will not need it here because it suffices for us to make use of the identity $K^0J^1 = K^1$.
%\end{remark}

%A tamely ramified cuspidal $G$-datum is \emph{generic} if in addition $\phi$ is generic in following sense \cite[\S15]{Yu2001}.  Note that $Y \to I+Y$ defines an isomorphism $\LieT_r/\LieT_{r+} \to T_r/T_{r+}$; we require that there exists a $G$-generic element $\Aphi \in \LieT_{-r}$ of depth $r$ such that for all $Y \in \LieT_{r}$ we have $\phi(I+Y) = \Psi(\Tr(XY))$. 

The idea of the construction of a supercuspidal representation of depth $r$ from such a quadruple $(\T,y,r,\phi)$ is to extend $\phi$ to a (uniquely determined) depth-$r$ representation $\rho$ of the compact open subgroup $\T\G_{y,s}$, whose compact induction to $\G$ is irreducible, and hence supercuspidal. It proceeds as follows.

% Begin by choosing a torus $\T$ from Table~\ref{Table:tori}.  For uniformity of notation, we can write
%$$
%\T = \left\{ t(a,b) := \mat{a & b\gam_1 \\ b\gam_2 & a} \mid a,b\in \ratk, a^2-b^2\gam_1\gam_2 = 1\right\},
%$$
%for $(\gam_1,\gam_2)$ as in the table.  Note that $\val(\gam_1^{-1}\gam_2) = 2y$, and the splitting field of $\T$ is obtained from $\ratk$ by adjoining a square root of $\rt = \gam_1\gam_2$.

We denote by $e$ the %Then the map $e=e_{s,r}\colon Y \mapsto I + Y$ defines 
natural isomorphisms of the abelian groups
$$
\LieT_{s+}/\LieT_{r+} \cong \T_{s+}/\T_{r+} \quad \textrm{and} \quad \g_{y,s+}/\g_{y,r+} \cong \G_{y,s+}/\G_{y,r+}.
$$
In matrix form one can approximate $e(X)$ by $I+X$, in the sense that these are congruent modulo the appropriate matrix groups.
 
Since the character $\phi$ of $\T$ has depth $r$, its restriction to $\T_{s+}$ is a character which factors through $\T_{s+}/\T_{r+}$.  All such characters are represented by elements of $\LieT_{-r}$, in the sense that there exists $\Aphi \in \LieT_{-r}\subset \g$ for which
%, which is given by %Then we may choose $\Aphi \in \LieT_{-(r+)} \subset \g$ so that for all $t\in \T_{s+}$, we have
\begin{equation} \label{E:Aphi}
\phi(t) = \Psi(\Tr(\Aphi e^{-1}(t))) \quad \textrm{for all $t \in \T_{s+}$}.
\end{equation}
The image of $\Aphi$ in $\LieT_{-r}/\LieT_{-s}$ is uniquely defined by this relation. 
The genericity of $\phi$ implies that there exists $\aphi \in \ratk$ such that $\Aphi = \aphi X_\T$  and $\val(\aphi\gam_1) = -r-y$.
This element $\Aphi$ also defines a character $\Psi_{\Aphi}$ of $\G_{y,s+}/\G_{y,r+}$ via
$$
\Psi_{\Aphi}(g) = \Psi(\Tr(\Aphi e^{-1}(g))) \quad  \textrm{for all $g \in \G_{y,s+}$}.
$$
Since $\phi$ and $\Psi_{\Aphi}$ agree on the intersection $\T_{s+}$ of their domains, together they give a unique well-defined character $\hat{\phi}$ of $\T\G_{y,s+}$, by setting $\hat{\phi}(tg) = \phi(t)\Psi_{\Aphi}(g)$ for any $t\in \T$, $g\in \G_{y,s+}$.  % This is the unique extension of $\phi$ and $\Psi_{\Aphi}$ to  $\T\G_{y,s+}$.

In case $\G_{y,s+} = \G_{y,s}$, we are done.  This occurs both when $\T$ is ramified, and when $\T$ is unramified and $r$ is an odd integer, since in these cases  $s$ is a fraction for which $\lrc{s}=\lrc{s+}$ and $\lrc{s\pm y} = \lrc{(s+) \pm y}$.   We thus set $\rho = \hat{\phi}$, which is a representation of $\T\G_{y,s}$ (of degree $1$ and depth $r$).

Otherwise, that is, when $\T$ is unramified and $r$ is even, we have the following lemma.  %then J.~K.~Yu defines $\rho$ as in the following lemma.
%specifies a canonical construction of a representation $\rho$ of $\T\G_{y,s}$ (of degree $q$) extending $\hat{\phi}$, and which satisfies the key properties in the following lemma.

\begin{lemma}[Yu] \label{L:Yu}
With the notation above, suppose $\T$ is unramified and $r$ is even.  Then there exists a canonical construction of an irreducible representation $\rho$ of $\T\G_{y,s}$ (of degree $q$) extending $\hat{\phi}$, satisfying:
\begin{enumerate}
\item $\Res_{\ZZ\T_{0+}}\rho$ is $\phi$-isotypic;
\item $\Res_{\G_{y,s+}}\rho$ is $\Psi_\Aphi$-isotypic.
\end{enumerate}
\end{lemma}

The construction of $\rho$, via the Weil representation, is central to \cite{Yu2001}, and we do not repeat it here.  Our lemma is a very slight generalization in that it includes the behaviour of the center.

\begin{proof}[Sketch of proof]
Fix $\T$ and define $g$ as in \eqref{E:g}.  Yu defines the subgroup $J^1$ (respectively, $J^1_+$) as the intersection with $\G$ of the subgroup of $\GG(\extk)$ generated by $\TT(\extk)_r$ and $(\GG_{\pm \alpha}(\extk)_{y,s})^{g}$ (respectively,  $(\GG_{\pm \alpha}(\extk)_{y,s+})^{g}$).   Then one has the identities $\T J^1=\T\G_{y,s}$ and $\T J^1_+=\T\G_{y,s+}$.

In \cite[\S11]{Yu2001}, Yu gives a canonical construction of a representation $\tilde{\phi}$ of $\T \ltimes J^1$, which depends only on $\Res_{\T_r}\phi$, as a pullback of a symplectic action such that
\begin{itemize}
\item  $\Res_{\T_{0+}\ltimes 1}\tilde{\phi}$ is $1$-isotypic, and
\item  $\Res_{1\ltimes J^1_+}\tilde{\phi}$ is  $\Psi_\Aphi$-isotypic.
\end{itemize}
By \cite[Prop 11.4]{Yu2001}, this symplectic action is given by conjugation, so the center $\ZZ$ acts trivially.  Hence  $\Res_{\ZZ\T_{0+}\ltimes 1}\tilde{\phi}$ is also $1$-isotypic.  Using that $\T\cap J^1 \subseteq \T_{r}$, one sees directly that the formula $\rho(tj) = \phi(t)\tilde{\phi}(t,j)$, for $(t,j)\in \T \ltimes J^1$, is well-defined.  It follows from the above that this representation $\rho$ of $\T\G_{y,s}$ has the desired properties.
\end{proof}

The following result is well-known for $\G=\SL_2(\ratk)$.  It is a special case of the general results about positive-depth tamely ramified supercuspidal representations given by the combination of \cite[Prop 4.6]{Yu2001},  \cite[Thm 19.1]{Kim2007} and \cite[Thm 6.7]{HakimMurnaghan2008}.

\begin{proposition}
Let $\rho = \rho(\T,y,r,\phi)$ be as above.
The compactly induced representation 
\begin{equation} \label{E:posdepth}
\cind_{\T\G_{y,s}}^\G \rho
\end{equation} 
is a supercuspidal representation of $\G$ of depth $r$, and all positive-depth supercuspidal representations of $\G$ arise in this way.  Moreover, two such representations are equivalent if and only if the pairs $(\T,\phi)$ occuring in their defining data are $\G$-conjugate.
\end{proposition}

\subsection{Representations of $\SL_2(\R)$} \label{SS:repR}

In his thesis \cite{Shalika1967}, J.~Shalika constructed all the irreducible representations of $\K=\SL_2(\R)$.  The ones we wish to recall here were those he called ``ramified representations'', and are constructed via Clifford theory.  Our notation here diverges from \cite{Shalika1967} in that: we use the depth, rather than the conductor, as the index (so our indices are off by one); and we use the single fixed quasi-character $\Psi$ and elements of the Lie algebra of negative depth, rather than a collection of additive characters $\eta_k$, $k\geq 0$, and elements of depth zero, in the parametrization.  

Let $d \in \Z_{>0}$ and let $u,v\in \ratk$ be such that $\val(v) > \val(u) = -d$.  Set 
\begin{equation}\label{E:xuv}
X = X(u,v)= \mat{0& u\\v&0} \in \g_{0,-d}.
\end{equation} 
Then the function $g \mapsto \Psi_X(g) = \Psi(\Tr(X(g-I)))$ defines a character of the group
$$
\Gmess = \G_{0,d/2}\cap \G_{\frac12,d/2}= \left\{ \mat{\U_{\lrc{d/2}} & \PP^{\lrc{d/2}} \\ \PP^{\lrc{(d+1)/2}} & \U_{\lrc{d/2}}} \right\}  \cap \G.
$$
%via
%$$
%\Psi_X(g) = \Psi(\Tr(X(g-I))) \quad \forall g\in \Gmess.
%$$
Note that $u$ and $v$ are uniquely determined by the character $\Psi_{X}$ only modulo $\PP^{\lrc{-d/2}}$ and $\PP^{\lrc{(-d+1)/2}}$, respectively.
The normalizer of $\Psi_X$ in $\K$ is $T(X)\Gmess$, where $T(X)$ is the centralizer of $X$ in $\K$, namely
$$
T(X) = \left\{ t(a,b)= \left.\left[ \begin{matrix} a & b \\ bu^{-1}v & a \end{matrix} \right] \right| a,b\in \R, a^2-b^2u^{-1}v =1 \right\} = \T_{1,u^{-1}v}.
$$
Then Shalika proved the following \cite[Theorems 4.2.1 and 4.25]{Shalika1967}.

\begin{theorem}[J.~Shalika]  \label{T:Shalika}
With the notation above, let $\theta$ be a character of $T(X)$ which coincides with $\Psi_X$ on the intersection $T(X) \cap \Gmess$.  Write $\Psi_{\theta,X}$ for the unique character of $T(X)\Gmess$ which extends $\theta$ and $\Psi_X$. Then 
$$
\Sh_d(\theta,X) := \Ind_{T(X)\Gmess}^\K \Psi_{\theta,X}
$$
is an irreducible representation of $\K$ of depth $d$ and of degree $\frac12 q^{d-1}(q^2-1)$.
For varying $X$ and $\theta$ as above, these representations exhaust all irreducible representations of $\K$ of this depth and degree.
\end{theorem}

\begin{proof}[Comments on the proof]
Since $\Psi_{\theta,X}$ is an extension of the character $\Psi_X$ to its normalizer, by an argument in Clifford theory, one deduces that $\Sh_d(\theta,X)$ is irreducible.  Since $\Psi_X$ is trivial on $\G_{0,d+}$ and nontrivial on $\G_{0,d}$, we deduce that the depth of $\Sh_d(\theta,X)$ as a representation of $\K = \G_{0}$ is $d$. Its degree is calculated directly.  

For the last statement, note that for all $g\in \K$ we have $\Sh_d(\theta,X)\cong \Sh_d(\theta^g,X^g)$. Shalika lists  representatives of all $\K$-orbits in $\g_{0,0}$ in \cite[Lemma~4.2.2]{Shalika1967}.  Our set of elements $\{X(u,v)\mid u\in \PP^{-d}\setminus\PP^{-d+1}, v\in \PP^{-d+1}\}$ was chosen to meet all those orbits corresponding to Shalika's ramified representations.  Shalika shows that these exhaust all irreducible representations of $\K$ of the given depth and degree in \cite[\S4.3]{Shalika1967}.
\end{proof}

\section{Restriction to $\K$} \label{S:restriction}

One of the main tools for decomposing restrictions of induced representations is Mackey theory. Applying versions of Frobenius reciprocity \cite[1.5(33)]{Cartier1979} and Mackey theory  for compactly induced representations \cite{Kutzko1977} yields the following result.

\begin{lemma} \label{L:Hansen}
Let $G$ be the $\ratk$-points of a linear algebraic group defined over $\ratk$.
Let $H$ be a compact-mod-center subgroup of $G$ and $\rho$
a smooth representation of $H$ such that the compactly induced representation
$\pi = \cind_H^G \rho$ is admissible.  Let $K$ be a compact open subgroup of $G$.  Then
the restriction of $\pi$ to $K$
decomposes into a direct sum of (not necessarily
irreducible) representations induced from subgroups of $K$ as
\begin{equation} \label{E:decomp}
\Res_K \cind_H^G \rho \cong \bigoplus_{\alpha \in K\backslash G / H} \Ind_{K \cap H^\alpha}^K \rho^\alpha.
\end{equation}
\end{lemma}

Let us find representatives of these double cosets and the depths and degrees of these Mackey components in each of the depth-zero and positive depth cases.

\subsection{Depth-zero case: double coset representatives, depth and degree}

From the Cartan decomposition of $\G$, and a short calculation, we see that a set of double coset representatives for either $\K \backslash \G / \K$ or $\K \backslash \G / \K^\eta$ is
\begin{equation} \label{E:doublecosetrepsK}
\left\{\left. \alp^t \doteq \mat{\p^{-t} & 0 \\ 0 & \p^t} \right| t \geq 0 \right\}.
\end{equation}
%\end{lemma}
Thus by Lemma~\ref{L:Hansen} the depth-zero supercuspidal representations decompose as
$$
\Res_{\K}\cind_{\K}^{\G} \sig \cong \bigoplus_{t \geq 0} \Ind_{\K \cap \K^{\alp^t}}^\K \sig^{\alp^t} \quad \textrm{and} \quad 
\Res_{\K}\cind_{\K^\eta}^{\G} \sig^\eta \cong \bigoplus_{t \geq 0} \Ind_{\K \cap \K^{\alp^t\eta}}^\K \sig^{\alp^t\eta}.
$$
For each $t>0$ the element $z_t=\p^tI$ satisfies $\alp^t = z_t\eta^{2t}$; since $z_t$ lies in the center of $\GL_2(\ratk)$, the conjugated representations $\sig^{\alp^t}$ and $\sig^{\eta^{2t}}$ are equal.  Furthermore, for any $t>0$, $\K \cap \K^{\eta^t} = B\K_{t}$, where $B$ denotes the set of upper triangular matrices in $\K$; that is, elements of $B\K_t$ can be represented by matrices $(a_{ij}) \in \K$ such that $a_{21}\in \PP^t$.  Consequently, the above decompositions may be rewritten as
\begin{equation} \label{E:depthzerodecomp}
\Res_{\K}\cind_{\K}^{\G} \sig \cong \sig \oplus \bigoplus_{t > 0} \Ind_{B\K_{2t}}^\K \sig^{\eta^{2t}} \quad \textrm{and} \quad 
\Res_{\K}\cind_{\K^\eta}^{\G} \sig^\eta \cong \bigoplus_{t \geq 0} \Ind_{B\K_{2t+1}}^\K \sig^{\eta^{2t+1}}.
\end{equation}

\begin{lemma} \label{L:depth}
Let $\sig$ be a depth-zero irreducible cuspidal representation of $\K$.  Then for any $d\geq 0$, the maximum depth of any irreducible component of $\Ind_{B\K_d}^\K \sig^{\eta^d}$ is $d$.
\end{lemma}

\begin{proof}
%Let $t \geq 0$.  Then $\K \cap \K^{\alp^t} = B\K_{2t}$, where $B$ denotes the subgroup of upper triangular matrices in $\K$.  
When $d=0$ the summand is simply $\sig$, which has zero depth by hypothesis.  Let $d>0$. Note that the maximum depth of an irreducible component of an induced representation $V$ of $\K$ is the least integer $n \geq 0$ such that $V = V^{\K_{n+1}}$, equivalently, such that the intersection of $\K_{n+1}$ with the inducing subgroup lies in the kernel of the representation being induced. Let $B^{op}$ denote the subgroup of lower triangular matrices in $\K$.

Given $\sig$ a nontrivial irreducible representation of $\K$ of depth zero, we note that the value of $\sig^{\eta^d}$ on elements of $B\K_d$ is determined by that of $\sig$ on $B^{op}$.  That is, for $a = (a_{ij})\in B\K_d$
%, so that $a_{ij}\in \R$ and $a_{21}\in \PP^d$, 
we have 
\begin{equation} \label{E:match}
\sig^{\eta^d}(a) = \sig\left(a^{\eta^{-d}}\right) = \sig\left(\mat{a_{11}&a_{12}\p^{d}\\ a_{21}\p^{-d} & a_{22}}\right) 
=  \sig\left(\mat{a_{11}&0\\ a_{21}\p^{-d} & a_{22}}\right)
\end{equation}
since $a_{12}\p^{d}\in \PP$ and $\sig$ is trivial on $K_1$.
If $n>d$ then $\K_n \cap B\K_{d} \subseteq \ker(\sig^{\eta^d})$ since in this case $a \equiv I \mod \PP^n$ and so $a^{\eta^{-d}} \equiv I$ mod $\PP$.  Conversely, since the subgroup $U^{op} = \left\{\smat{1 & 0 \\ c & 1} \mid c\in \R^\times \right\}$ is contained in no proper normal subgroup of $\K$ but  $(U^{op})^{\eta^d} \subseteq \K_d \cap B\K_{d}$, we conclude that $ \K_d \cap B\K_{d} \not\subset \ker(\sig^{\eta^d})$.  Thus the maximal depth is $d$.  %Thus the least $n$ such that $\K_n$ acts trivially on this representation is $n=d$. % if and only if $\sig$ is the trivial representation.  %Thus $\Ind_{B\K_d}^\K \sig^{\eta^d}$ has depth $d$.
%On the other hand, we have  $\K \cap \K^{\alp^t\eta} = B\K_{2t+1}$, and it follows as above that the component $\Ind_{\K \cap \K^{\alp^t\eta}}^\K \sig^{\alp^t\eta}$ has depth $2t+1$.
\end{proof}

Finally, note that for any $d>0$,  $B\K_d$  has index $(q+1)q^{d-1}$ in $\K$. Hence %Similarly, for any $t \geq 0$, $\K \cap \K^{\alp^t\eta}=B\K_{2t+1}$, which has index $(q+1)q^{2t+1}$ in $\K$.  Hence we have
\begin{equation}\label{E:depthzerodegree}
\deg\left(\Ind_{B\K_d}^\K \sig^{\eta^d}\right) = (q+1)q^{d-1} \deg(\sig),
\end{equation}
%and 
%\begin{equation} \label{E:degreesigma2}
%\deg\left(\Ind_{\K \cap \K^{\alp^t\eta}}^\K \sig^{\alp^t\eta}\right) = \deg(\sig)(q+1) q^{2t+1}.
%\end{equation}
where if $\sig \in \{\sig_0^\pm\}$ then $\deg(\sig) = (q-1)/2$, but for all other $\sig$, we have $\deg(\sig)=q-1$.

In Section~\ref{S:depthzero}, we will use this as a starting point to give the complete decomposition of depth-zero supercuspidal representations into irreducible $\K$-representations.

\subsection{Positive depth case: double coset representatives}
Let us now turn to the positive-depth case.  
To determine double coset representatives for $\K \backslash \G / \T\G_{y,s}$, for the various triples $(\T, y, s)$ that  occur in the construction, we begin with the following observation.  

\begin{lemma} \label{L:doublecosets}
Let $\T$ be one of the tori in Table~\ref{Table:tori} such that $\T \subset \K$.
%.  When $\val(\gam_1)=-1$, we have $\T \subset \K^\eta$, and $\T \subset \K$ in all other cases.  
Choose $x,y \in \R$ satisfying $x^2-y^2\ep = \ep$.  Set
$$
\dcrep = \begin{cases}
\ecrep = \mat{x & y \\ y & \ep^{-1}x} & \textrm{if $\T$ is unramified, and}\\
%\mat{x & y(\p)^{-1} \\ y\p & x\ep^{-1}} & \textrm{if $i=2$}\\
\w = \mat{0 & 1 \\ -1 & 0} & \textrm{if $\T$ is ramified.}\
%
%E^\eta = \mat{x & y(\p)^{-1} \\ y\p & x\ep^{-1}} 
%& \textrm{if $\T = T_1^\eta$.}
\end{cases}
$$
Then $\Lambda(\T) = \{ I, \dcrep\}$ is a set of representatives of $B^{op}\backslash \K / \T$.  Similarly, $\Lambda(\T_{1,\ep}^\eta) = \{I,\ecrep^\eta\}$ is a set of representatives of $B^{op}\backslash \K^\eta / \T_{1,\ep}^\eta$.
\end{lemma}

Note that if $-1=t^2$, then one can take $(x,y)=(0,t)$; otherwise, we solve $x^2 + y^2 = -1 = \ep$.

\begin{proof}
Write $\T = \T_{\gam_1,\gam_2}$ and for each $g=(g_{ij})\in \K$ set $N(g) = g_{11}^2-\gam_2\gam_1^{-1}g_{12}^2$.  We assume $\T \subseteq \K$; the remaining case $\T_{1,\ep}^\eta \subset \K^\eta$ is similar.   Given $g \in B^{op}\T$, write $g=bt$ with $b = \smat{u&0\\v&u^{-1}}$ and $t = t(a,b)$; then $N(g)=u^2$.  Conversely, given $g \in \K$ such that $N(g) = u^2$ for some $u\in \R^\times$, we see directly that $t=t(g_{11}u^{-1}, g_{12}\gam_1^{-1}u^{-1}) \in \T$ and that  $g t^{-1} \in B^{op}$.  Thus the identity double coset $B^{op}\T$ consists of all $g\in \K$ with $N(g) \in (\R^\times)^2$.

In the ramified case, $\val(\gam_2\gam_1^{-1}) = 1$ so $N(g) \in (\R^\times)^2$ if and only if $g_{11} \in \R^\times$.  Thus if $g$ lies in the complement $\K \setminus B^{op}\T$, then $g_{11}\in \PP$ and $g_{12} \in \R^\times$.  Since $p\neq 2$ we can choose $u\in \R^\times$ such that  $N(g) = -\gam_2\gam_1^{-1}u^2$.   With this choice of $u$,  $t = t(g_{12}u^{-1}, g_{11}\gam_2^{-1}u^{-1}) \in \T$ and $gt^{-1}\w^{-1} \in B^{op}$.  Thus  $\K \setminus B^{op}\T = B^{op}\w \T$.

In the unramified case, $\gam_2\gam_1^{-1} = \ep$. 
% so $N(g) \in \R^\times$ for all $g\in \K$; by construction $N(\ecrep)=\ep$.  
Let $g\in \K \setminus B^{op}\T$; then necessarily $N(g)=u^2\ep$ for some $u\in \R^\times$.   To show that $g\in  B^{op}\ecrep\T$, we solve  $g=b\ecrep t$ for the unknown $t=t(c,d)$ by setting $b_{11}=u$ and $b_{12} = 0$.  The first row of this matrix equality gives a linear system with solution
$$
\mat{c\\d} = \frac{1}{u} \mat{x&-\ep y\\-y&x}\mat{g_{11}\\g_{12}}.
$$
Since this yields $\det(t(c,d))=1$, $t\in \T$ and we conclude that $g\in  B^{op}\ecrep\T$. 
% , we see that the identity coset $B^{op}\T$ is characterized as the set of all matrices $g = (g_{kl}) \in \K$ (respectively, $\K^\eta$ if $\T = \T_{\p^{-1},\ep\p}$) such that $g_{11}^2 - (\gam_2/\gam_1) g_{12}^2 = s^2$ for some $s\in \R^\times$.  
%Its complement consists of those $g$ for which $g_{11}^2 - (\gam_2/\gam_1) g_{12}^2 \notin (\R^\times)^2$.  
%In the unramified case, and for $g\in \K$ or $\K^\eta$, as the case is defined, the expression $g_{11}^2 - (\gam_2/\gam_1) g_{12}^2$ must evaluate to an element of $\R^\times$.  It follows that the complement of the identity coset is itself a double coset, which can be represented by any $g$ satisfying $g_{11}^2 - (\gam_2/\gam_1) g_{12}^2 = \R^\times \setminus (\R^\times)^2$.  
%Note that we have simply chosen $\dcrep_2 = \dcrep_1^\eta$.
%In the ramified case, $\val(\gam_2/\gam_1)=1$ so the identity double coset is more simply defined by the property that $g_{11} \in \R^\times$.  On its complement, therefore, the expression  $g_{11}^2 - (\gam_2/\gam_1) g_{12}^2$ must evaluate to $-r^2 \p$ for some $r\in \R^\times$, and such elements form one double coset.
%The lemma follows.
\end{proof}

%From Lemma~\ref{L:doublecosets} the determination of the desired double coset space is immediate.  

\begin{proposition}
%Let $\T$ be chosen from Table~\ref{Table:tori}.  
A set of representatives for the double coset space $\K \backslash \G / \T\G_{y,s}$ is
\begin{equation} \label{E:MT}
\Mu(\T) = \{ I, \alp^t\lambda \mid t > 0, \lambda \in \Lambda(\T) \}.
\end{equation}
%replace $\dcrep_1$ with $\dcrep_1^\eta$ instead.
\end{proposition}

\begin{proof}
Assume $\T \subset \K$; the case $\T=\T_{1,\ep}^\eta \subset \K^\eta$ is similar.  Write $\Lambda(\T) = \{I,\dcrep\}$.

Since $\K\backslash \G/\K$ is represented by $\{\alp^t\mid t\geq 0\}$, each double coset of $\K \backslash \G / \T\G_{y,s}$ can be represented by an element of the form $\alp^t\beta$, with $t \geq 0$ % $\alp^t$ runs over $\Sigma$ and 
and $\beta$ a representative of $(\K \cap \K^{\alp^{-t}})\backslash \K / \T\G_{y,s}$.
When $t=0$, $(\K \cap \K^{\alp^{-t}})=\K$ and so we take $\beta = 1$.  
For any $t > 0$, the group $\K \cap \K^{\alp^{-t}}$ contains the group $B^{op}$; also $\T\G_{y,s}$ contains $\T$.  So each double coset is a union of $B^{op}\backslash \K / \T$ double cosets, and Lemma~\ref{L:doublecosets} applies.  Since $s,t>0$, we can verify that  $B^{op}\T \equiv (\K \cap \K^{\alp^{-t}})\T\G_{y,s}$ modulo $\PP$, so $\dcrep \notin (\K \cap \K^{\alp^{-t}})\T\G_{y,s}$.  Hence $(\K \cap \K^{\alp^{-t}})\backslash \K / \T\G_{y,s}$ is also represented by $\{I, \dcrep\}$, as 
required.
\end{proof}

We conclude that with $\Mu(\T)$ as in \eqref{E:MT} we have
\begin{equation} \label{E:pddecomp}
\Res_{\K}\cind_{\T\G_{y,s}}^{\G} \rho \cong  \bigoplus_{\mu \in \Mu(\T)} \Ind_{\K \cap (\T\G_{y,s})^{\mu}}^\K \rho^{\mu}. 
\end{equation}
We work towards a more explicit description of these inducing subgroups.
Let $\T$ be a torus from Table~\ref{Table:tori} and $\mu = \alp^t\lambda \in \Mu(\T)$.  Set
 $$
\delta(\mu) = \begin{cases} 
2t-y & \textrm{if $y=\frac12$, $\lambda = \w$}\\
2t+y & \textrm{otherwise.}
\end{cases}
$$

\begin{lemma}\label{L:Tdeltamu}
Given $\T$ as in Table~\ref{Table:tori}, $\Aphi = \LieT\setminus \{0\}$ and $\mu \in \Mu(\T)$, we have
$$
T(\Aphi^\mu) = \K \cap \T^\mu = Z\T_{\delta(\mu)}^\mu 
$$
where $T(X)$ denotes the centralizer  of $X$ in $\K$, as in Theorem~\ref{T:Shalika}.  Furthermore, for any $m\in \Z_{> 0}$ we have $\K_m \cap \T^\mu = (\T_{\delta(\mu)+m})^\mu$.
\end{lemma}

\begin{proof}
That $\K \cap \T^\mu = T(\Aphi^\mu)$ follows directly, since $\T$ is the centralizer in $\G$ of $\Aphi$.  
For $m\geq 0$ we have $\K_m^{\mu^{-1}} = \G_{0,m}^{\mu^{-1}} = \G_{\mu^{-1}\cdot 0,m}$.   If $\mu = \alp^t$ then this group is $\G_{-2t,m}$ whereas if $\mu = \alp^t\w$ then it is $\G_{2t,m}$.  Comparing these matrix groups reveals that the intersection
 $\K_m^{\mu^{-1}}\cap\T$ is, up to centre when $m=0$,  the filtration subgroup $\T_{\delta(\mu)+m}$.
In the remaining cases, $\mu = \alp^t\dcrep$ does not necessarily normalize the apartment $\A$, but $\dcrep\in \G_y$ normalizes $\T$ and also its Moy-Prasad filtration subgroups.  Thus $\K_m^{\alp^{-t}}\cap\T^\dcrep = \K_m^{\alp^{-t}}\cap\T = \T_{\delta(\mu)+m}=\T_{\delta(\mu)+m}^\dcrep$.  The result follows.
%but perhaps not the apartment $\A$.  Nevertheless, for any $p>0$ we have $(\T_p)^\dcrep = (\T\cap\G_{y,p})^\dcrep = \T^\dcrep \cap \G_{y,p}$.  Writing $(\T_p)^\dcrep$ explicitly as in Table~\ref{Table:xt} allows us to conclude that $\K^{\mu^{-1}}_m\cap \T^\dcrep = \G_{y,p}\cap \T^\dcrep$ for $p-y=2t+m$, whence again the result.
\end{proof}

For reference we identify in  Table~\ref{Table:xt} the action of each $\lambda \in \Lambda(\T)$ on a basis element $X_\T$ of $\LieT$.  Note that $\T^\mu = \{aI+bX_\T^\mu \mid a,b\in \R\}\cap \G$.  Thus if $\mu=\alp^t\lambda$ and $X_\T^\lambda = X(u,v)$ then 
\begin{equation} \label{E:KT}
\K \cap \T^{\mu} = \left\{ \left.\mat{a & b \\ bu^{-1}v\p^{4t} & a}\right| a,b\in \R\right\} \cap \G = \T_{1,u^{-1}v\p^{4t}}.
\end{equation}
\begin{table}
\begin{center}
\renewcommand{\arraystretch}{1.5}
\begin{tabular}{|l|l|l|}
\hline
Torus $\T$ & $X_\T \in \LieT$ & Conjugate $X_\T^\dcrep$ \\
\hline
$\T_{1,\ep}$ & $X(1,\ep)$ & $X_\T^{\ecrep} = X(\ep,1)$\\
$\T_{\p^{-1},\ep \p}$ & $X(\p^{-1},\ep \p)$ & $X_\T^{\ecrep^\eta}= X(\ep \p^{-1}, \p)$\\ 
$\T_{\gam_1,\gam_2}$ & $X(\gam_1,\gam_2)$ & $X_\T^{\w} = X(-\gam_2,-\gam_1)$\\
\hline
$\T$ & $X_\T^{\lambda} = X(u,v)$ & $X_\T^{\alp^t\lambda} = X(u\p^{-2t},v\p^{2t})$ \\ 
\hline
\end{tabular}
\caption{Values of $X_\T^\mu$ for various tori $\T$ and $\mu = \alp^t\lambda \in \Mu(\T)$, with $\lambda \in \Lambda(\T)=\{1,\dcrep\}$.} \label{Table:xt}
\end{center}
\end{table}

Similarly, we can describe $\K\cap \G_{y,s}^\mu$ succinctly: 
when $y\in \{0,1\}$,  $\Lambda(\T) \subset \G_y$, so for any $\mu=\alp^t\lambda\in \Mu(\T)$, we have $\G_{y,s}^{\mu} = \G_{y,s}^{\alp^t}.$  As in the proof of Lemma~\ref{L:Tdeltamu}, we can easily write down $\G_{y,s}^{\alp^t}$ and $\G_{1/2,s}^{\alp^t\w}$.  In general, setting  $M = \max\{0,\lrc{s-\delta(\mu)}\}$, we have
\begin{equation} \label{E:KG}
\K \cap \G_{y,s}^{\mu}  = \left\{ \mat{\U_{\lrc{s}} & \PP^M \\ \PP^{\lrc{s+\delta(\mu)}} & \U_{\lrc{s}}} \right\} \cap \G.
\end{equation}

\begin{proposition} \label{P:split}
Let $\mu\in \Mu(\T)$.  Then
$
\K \cap (\T\G_{y,s})^{\mu} = (\K \cap \T^{\mu})(\K \cap \G_{y,s}^{\mu}).
$  Furthermore, if $\delta(\mu) > s$, then $\K \cap (\T\G_{y,s})^{\mu} = Z(\K \cap \G_{y,s}^{\mu})$.
\end{proposition}

\begin{proof}
Note that since $\T$ normalizes $\G_{y,s}$, $\K \cap \T^\mu$ normalizes $\K \cap \G_{y,s}^\mu$.  Let $\mu =\alp^t\lambda \in \Mu(\T)$.  If $t=0$ then since $s>0$ we see that for all $y\in \{0,\frac12,1\}$, $\G_{y,s}\subseteq \K$, and there is nothing to show.  If $t>0$ and $\lambda \neq \w$, then $\G_{y,s}=\G_{y,s}^\lambda$, so it suffices to prove that 
$\K \cap (\T^\lambda\G_{y,s})^{\alp^t} = 
(\K \cap (\T^\lambda)^{\alp^t})(\K \cap (\G_{y,s})^{\alp^t})$.  We use the explicit matrix forms, above.

Factor $g\in \T^\lambda\G_{y,s}$ as $g=uh$ with $u = (u_{ij}) \in \T^\lambda$ and $h = (h_{ij}) \in \G_{y,s}$.  If $\lrc{s-y}\geq 2t$ then $h_{12}\in \PP^{2t}$ and consequently $h^{\alp^t} \in \K$; thus $g^{\alp^t} \in \K$ if and only if $u^{\alp^t}\in \K$, and our factorization holds.  
Otherwise, note that if $g^{\alp^t}=(uh)^{\alp^t} \in \K$ then its $(1,2)$ matrix entry satisfies 
$$
  u_{11}h_{12}\p^{-2t} + u_{12}h_{22}\p^{-2t} \in \R.
$$
As $u_{11},h_{22}\in \R^\times$, we deduce $\val(u_{12})= \val(h_{12})$, which by definition of $h$ is at least $\lrc{s-y}$.  It follows that
%which happens only if $\val(u_{12}) \geq \val(h_{12}) \geq \lrc{s-y}$.  In these circumstances, 
 $u \in \T^\lambda\cap Z\G_{y,s}$.  
%Thus for some $i\in Z=\{\pm I\}$, $iu \in  \G_{y,s}$.  
We can thus refactor $g$ as $ih'$ with $i\in Z\subset \T^\lambda$ and $h'\in \G_{y,s}$.  Since $i=i^{\alp^t} \in \K$ we deduce that $(h')^{\alp^t}\in \K$, as required.

The case $\lambda = \w$ follows by replacing $G_{y,s}$ with $G_{y,s}^{\w}$ and $\lrc{s-y}$ with $\lrc{s+y}$.
\end{proof}

Note that this proposition does not hold if we replace the factorization $\T\G_{y,s}$ with the factorization $\T J^1$ referred to in the proof of Lemma~\ref{L:Yu}.

\subsection{Positive depth case: depths and degrees of the Mackey components} \label{SS:depth}
%Let $(\pi,V)$ be one of the (possibly reducible) induced $\K$-representations occuring in \eqref{E:pddecomp}.  

\begin{proposition} \label{P:depth}
Let $\rho = \rho(\T,y,r,\phi)$ and $\mu\in \Mu(\T)$.  Then the maximum depth $d$ of any irreducible $\K$-component of the representation $\Ind_{\K \cap (\T\G_{y,s})^{\mu}}^\K \rho^{\mu}$
is $d=r + \delta(\mu)$.
\end{proposition}

\begin{proof}
By Lemma~\ref{L:Yu}, we have that $\ker(\rho) \supseteq \G_{y,r+}$, and that $\Res_{\ZZ T_{0+}}\rho$ is $\phi$-isotypic, where $\phi$ has depth $r$.  (In fact, one can explicitly describe $\ker(\rho)$ as the subgroup generated by $\T_{r+}$ and $J_+^1$ %\cite[Section 3.1]{HakimMurnaghan2008}, 
but this is more than is needed here.)  It follows that $\K \cap \G_{y,r+}^\mu \subseteq \ker(\rho^\mu)$; from \eqref{E:KG} we see it contains $\K_n$ for all $n>r+\delta(\mu)$, so the maximal depth is at most $r+\delta(\mu)$.

To show that $\K_{r+\delta(\mu)} \not\subset  \ker(\rho^\mu)$, write $\mu=\alp^t\lambda$ and let  $\Aphi = \aphi X_\T \in \LieT_{-r}$ represent $\phi$, which is $\G$-generic of depth $r$.  Then $\val(\aphi \gam_1) = -r-y$.  First suppose $\lambda=1$.  For any $c\in \R^\times$, set $g_c = \smat{1& 0 \\ c\p^{r+y} & 1}$.  Then $\rho^{\alp^t}(g_c^{\alp^t}) = \rho(g_c)=\Psi_\Aphi(g_c)I_{\deg(\rho)} = \Psi(\aphi \gam_1 c \p^{r+y}))I_{\deg(\rho)}$ is nontrivial.  
Since $g_c^{\alp^t} \in \K_{r+y+2t}$, we conclude that $\rho^{\alp^t}$ has (a component of) depth $r+\delta(\alp^t)$.  

Now suppose $\lambda \neq 1$.  When $y=0$, $\lambda=\ecrep\in \K$ normalizes $\K_r$ so $\rho^\mu((g_c^{\ecrep^{-1}})^\mu) = \rho(g_c^{\ecrep^{-1}})$ is nontrivial and $(g_c^{\ecrep^{-1}})^\mu = g_c^{\alp^t} \in \K_{r+2t}$.  When $y=1$, the element to consider is instead $g_c^{(\ecrep^\eta)^{-1}}$.  When $y=\frac12$, and $\lambda=\w$, replace $g_c$ by $g_c'=\smat{1& c\p^{r-y} \\ 0 & 1}$.  We verify that $\rho^\mu({g_c'}^{\mu})=\rho(g_c')$ is nontrivial and ${g_c'}^{\mu} \in \K_{r+2t-y}$.  In all cases we conclude that $\rho^\mu$ has (a component of) depth $r+\delta(\mu)$.
\end{proof}

Now we determine the degree of each $\K$-representation occuring in \eqref{E:pddecomp}.

\begin{proposition} \label{P:degree}
Let $\rho = \rho(\T,y,r,\phi)$ and $\mu =\alp^t\lambda \in \Mu(\T)$.
If $t=0$ and $y=0$ then $\deg\left(\Ind_{\T\G_{0,s}}^\K \rho\right) = (q-1)q^r$.  In all other cases, setting $d= r+\delta(\mu)$ we have
$$
\deg\left(\Ind_{\K \cap (\T\G_{y,s})^{\mu}}^\K \rho^{\mu}\right) = \frac{q^2-1}{2}q^{d-1}.
$$
%where $d=r+\delta(\mu)$.% is the maximum depth of a component of this representation of $\K$, as in Corollary~\ref{C:depth}.
\end{proposition}

\begin{proof}
%Since $\deg(\rho^\lambda) = \deg(\rho) \in \{1,q\}$, depending on whether or not $\G_{y,s} = \G_{y,s+}$, 
Let $\mu =\alp^t\lambda \in \Mu(\T)$.  % Assume first $\lambda \neq \w$.
The degree of the induced representation is given by $\deg(\rho)$ times the index of $\K \cap (\T\G_{y,s})^{\mu}$ in $\K$.
Using Proposition~\ref{P:split}, and that $(\K \cap \T^\mu)(\K \cap \G_{y,s}^\mu) \supset \G_{y,s+2t+1}$, we find by the second isomorphism theorem
$$
[\K:\K \cap (\T\G_{y,s})^{\mu}]=\frac{[\K:\G_{y,s+2t+1}]}{[\K \cap \G_{y,s}^\mu: \G_{y,s+2t+1}][(\K \cap \T^\mu) : (\K \cap \T^\mu \cap \G_{y,s}^\mu)]}.
$$
From our explicit descriptions in  \eqref{E:Gyr} and \eqref{E:KG} (and noting that $\G_{y,s+2t+1}\subseteq \K_1=\G_{0,1}$), we determine
\begin{align*}
[\K : \G_{y,s+2t+1}] &= (q^2-1)q^{6t+\lrc{s} + \lrc{s+y} + \lrc{s-y} +1}\\
[\K \cap \G_{y,s}^\mu : \G_{y,s+2t+1}]&=q^{6t+\lrc{s-y}+\lrc{s+y}-M-\lrc{s+\delta(\mu)}+3}.
\end{align*}
%and, letting $M = \max\{0,\lrc{s-y}-2t\}$ as in \eqref{E:KG1}, we have $[\K \cap \G_{y,s}^\mu \colon \G_{y,s+2t}]=q^{4t+\lrc{s-y}-M}$.
%To count $[(\K \cap \T^\mu) \colon (\K \cap \T^\mu \cap \G_{y,s}^\mu)]$, we use the description of $\K \cap \T^\mu$ from \eqref{E:KT}.  
When $t=0$ and $y=0$, we have $[(\K \cap \T^\mu) : (\K \cap \T^\mu \cap \G_{y,s}^\mu)] = [\T: \T\cap \G_{0,s}] = \vert \TT(\resk) \vert [\T_1:\T_s] = (q+1)q^{\lrc{s}-1}$.  Since $\deg(\rho)=q$ exactly when $r=2\lrc{s}$, and is $1$ when $r = 2\lrc{s} -1$, we deduce the total degree is
$(q-1)q^{2\lrc{s}-1}\deg(\rho) = (q-1)q^r$.

When $t>0$ or $y\neq 0$, then $\K \cap \T^\mu$ is contained in the standard Iwahori subgroup.  With notation as in \eqref{E:KT}, we see that for every $b \in \R$, there exist exactly two choices for $a$ such that $t(a,b)\in \K\cap\T^\mu$.  Furthermore, if $b\in \PP^M$ then 
$a \in \pm \U_{2\delta(\mu)+2M} \subset \pm \U_{\lrc{s}}$, so that $t(a,b)\in Z\G_{y,s}^\mu$ if and only if $b\in \PP^M$.  Thus
$$
[(\K \cap \T^\mu) : (\K \cap \T^\mu \cap \G_{y,s}^\mu)]=2q^M.
$$
% so the number of cosets is just twice the number of choices for $b \in \R/\PP^{M}$, that is, $2q^M$.
Consequently
$$
\deg\left(\Ind_{\K \cap (\T\G_{y,s})^{\mu}}^\K \rho^{\mu}\right) = \frac12 \deg(\rho) (q^2-1)q^{\lrc{s} + \lrc{s+\delta(\mu)} - 2}.
$$
When $r$ is an even integer, $\deg(\rho)=q$ and $y$ and $s$ are integers, so the expression simplifies to $\frac12(q^2-1) q^{r+y+2t-1}$, as required.   Otherwise, we have $\deg(\rho)=1$, and either $r$ is an odd integer and $y$ is an integer, or else $r$ and $y$ are half-integers.  In either case, $\lrc{s}+\lrc{s\pm y} = 2s\pm y+1$.  Thus we again recover the desired formula.
% and either $r,y \in \Z$ and $r$ is odd, or $r\in \frac12 + \Z$, $y=\frac12$ and $\mu = \alpha^t$.  One verifies directly that $2t+ \lrc{s} + \lrc{s+y} - 2 = r+y+2t-1$ in these cases.
%Finally, if $t>0$ and $\lambda=\w$, then repeating the argument instead yields
%$$
%\deg\left(\Ind_{\K \cap (\T\G_{y,s})^{\alp^t\w}}^\K \rho^{\alp^t\w}\right) = \frac12 (q^2-1)q^{2t+ \lrc{s} + \lrc{s-y} - 2} = \frac12 (q^2-1)q^{r-y+2t-1},
%$$
%as required.
\end{proof}

% \begin{remark}
%One could at this point identify all these subrepresentations as occuring in the restriction of certain irreducible representations of $\GL_2(\R)$ arising in \cite{Hansen}, whence their irreducibility, by an argument applied in the next section.  For the positive depth case, we instead rephrase Hansen's arguments for the $\SL_2$ case to obtain several resuls, including irreducibility.
%Maybe have this after decomp?
%\end{remark}
 
\section{Branching rules: depth-zero case}  \label{S:depthzero}

We reprise the notation for depth-zero supercuspidal representations.  In particular $\om$ denotes a nontrivial character of $\ker(N)$, the kernel of the norm map of a quadratic extension of $\resk$, and $\sig = \sig(\om)$ denotes both the corresponding representation of $\SL_2(\resk)$, and its inflation to $\K$.  We write $\om_0$ for the unique character of order $2$, for which $\sig_0 = \sig(\om_0)$ decomposes into the two cuspidal representations $\sig_0^\pm$.  %Recall that the central character of $\sig$ coincides with the restriction of $\om$ to the subgroup $\{\pm 1\}$.

%To simplify \eqref{E:dzdecomp} and \eqref{E:dzdecomp2}, note that for each $t>0$ the element $z_t=\p^tI$ in the center of $\GL_2(\ratk)$ gives $\alp^t = z_t\eta^{2t}$ and $\alp^t\eta = z_t\eta^{2t+1}$, whence the conjugated representations are equivalent.  Furthermore, for any $t>0$, $\K \cap \K^{\eta^t} = B\K_{t}$, where $B$ denotes the set of upper triangular matrices in $\K$.
%Consequently, the $\K$-representations occuring in the decomposition \eqref{E:dzdecomp} or \eqref{E:dzdecomp2} of any depth-zero supercuspidal representation all have the form
%\begin{equation} \label{E:dzcase}
%\Ind_{\K \cap \K^{\eta^t}}^\K \sig^{\eta^t} \cong \Ind_{B\K_t}^{\K} \sig^{\eta^t}
%\end{equation}
%with the even values of $t>0$ arising from the decomposition of supercuspidals induced from $\K$ and the odd values of $t>0$ arising from those from $\K^\eta$.  By Lemma~\ref{L:depth}, the depth of the representation in \eqref{E:dzcase} is $t$ and by \eqref{E:degreesigma1} and \eqref{E:degreesigma2} its degree is $\frac12(q^2-1)q^t$ if $\sig \in \{\sig_0^\pm\}$ and $(q^2-1)q^t$ otherwise.

%Let $\sig=\sig(\om)$, $\om \neq 1$.  

%On the other hand, we know that the representations $\Sh_t(\rho,X)$, for suitable choices of $\rho$ and $X$, are irreducible and of depth $t$ and degree $\frac12(q^2-1)q^t$.  

We begin by showing that for each $d>0$ and each $\sig = \sig(\om)$ the representation
$
\Ind_{B\K_d}^\K \sig^{\eta^d}
$
is independent of the choice of $\sig$ up to its central character, by showing that this is true of the restriction to $B\K_d$ of $\sig^{\eta^d}$.

\begin{lemma} \label{L:resbkd}
Let $d>0$.  Let $\om_1,\om_2$ be two nontrivial characters of $\ker(N)$ and $\sig_i=\sig(\om_i)$, $i\in \{1,2\}$,  the corresponding representations of $\K$.   Let $\tau$ denote the trivial extension to $B\K_d$ of a character of the diagonal torus of $\K$ such that $\tau(-1)=-1$.  
Then we have
\begin{equation} \label{E:E:resbkd}
\Res_{B\K_d}\sig_1^{\eta^d} \cong \begin{cases}
\Res_{B\K_d}\sig_2^{\eta^d} & \textrm{if $\om_1(-1)=\om_2(-1)$;}\\
\tau \Res_{B\K_d}\sig_2^{\eta^d} &\textrm{otherwise.}
\end{cases}
\end{equation}
Furthermore, this restriction decomposes as a direct sum of two inequivalent irreducible subrepresentations.
\end{lemma}

\begin{proof}
We saw in \eqref{E:match} that the restriction of $\sig_i^{\eta^d}$ to $B\K_d$ is determined by the restriction of $\sig_i$ to $B^{op}$.  It thus suffices to show that
%\begin{equation} \label{E:match}
%\sig_i^{\eta^t}\left(\mat{a & b \\ c\p^{t} & d}\right) =  \sig_i\left( \mat{a & b\p^t \\ c & d} \right)=
%\sig_i\left( \mat{a & 0 \\ c & d} \right).
%\end{equation}
%Writing $B^{op}$ for lower triangular subgroup of $\K$, let us show that
\begin{equation} \label{E:equivBop}
\Res_{B^{op}}{\sig_1}\cong 
\begin{cases}
\Res_{B^{op}}\sig_2 & \textrm{if $\om_1(-1) = \om_2(-1)$ and}\\
\tau \Res_{B^{op}}\sig_2 & \textrm{otherwise,}
\end{cases}
\end{equation}
and that these each decompose as a direct sum of two inequivalent irreducible representations. Since these representations factor through the finite group quotient $\SL_2(\resk) \cong \K/\K_1$, it suffices to compare their characters.  Write $B^{op}$ also for its image in $\SL_2(\resk)$.  The character $\chi_i$ of $\Res_{B^{op}}\sig_i$ is given on elements of $B^{op}$ by \cite[\S 15, Table 2]{DigneMichel1991} 
$$
\chi_i \left( \mat{a & 0 \\ c & a^{-1}} \right) 
= \begin{cases}
(q-1)\om_i(a) & \text{if $a = \pm 1$, $c=0$};\\
-\om_i(a) & \text{if $a = \pm 1$, $c \neq 0$};\\
0 & \text{otherwise}.
\end{cases}
$$
The character of $\tau\Res_{B^{op}}\sig_2$ is $\tau\chi_2$.

%Let $\varphi \in \{1, \tau\}$ and denote by $I_{B^{op}}(\sig_1, \varphi\sig_2)$ the corresponding intertwining number.
Setting $g_c = \smat{1 & 0 \\ c & 1}$ and noting that $\chi_i$ is real-valued, we calculate the intertwining number between $\chi_1$ and $\chi_2$ to be
\begin{align*}
I(\chi_1,\chi_2)&= \frac{1}{\vert B^{op} \vert} \sum_{g \in B^{op}}\chi_1(g) \overline{\chi_2(g)} \\
&= \frac{1}{q(q-1)} \left( \sum_{a\in \{\pm 1\}, c=0} \chi_1(a)\chi_2(a) + \sum_{a\in \{\pm1\},c \neq 0} \chi_1(ag_c)\chi_2(ag_c)\right)\\
&=  \frac{1}{q(q-1)}\left(1 + \om_1(-1)\om_2(-1)\right) \left( (q-1)^2 + (q-1) \right)\\
&= \begin{cases}
2 & \text{if $\om_1(-1) = \om_2(-1)$; } \\ 
0 & \text{otherwise.}
\end{cases}
\end{align*}
It follows that for each $i$, $\Res_{B^{op}}\sig_i$ decomposes as a direct sum of two inequivalent irreducible representations of $B^{op}$, and that $\Res_{B^{op}}\sig_1$ is equivalent to $\Res_{B^{op}}\sig_2$ exactly when their central characters coincide.  On the other hand, when $\sig_1$ and $\sig_2$ have opposite central character, we may argue as above that $I(\chi_1, \tau\chi_2)=2$, which completes the proof. % \eqref{E:equivBop}.  %Finally, \eqref{E:equivBop} implies  via \eqref{E:match} that the inducing representations of \eqref{E:equivdz} are equivalent exactly when the central characters coincide, whence the result.
\end{proof}

When $\om = \om_0$, the unique nontrivial quadratic character of $\ker(N)$, we know that $\sig_0 = \sig_0^+\oplus \sig_0^-$.  Applying Lemma~\ref{L:resbkd} with $\om_2=\om_0$ therefore yields an explicit description of the decomposition of $\Res_{B\K_d}\sig_1^{\eta^d}$ into irreducible subrepresentations.  
Consequently, for $\theta$ a character of $Z$ and $d>0$, we define
%and consequently, a decomposition of $\Ind_{B\K_d}^\K \sig^{\eta^d}$ for any $\sig$.    Let $\theta$ denote the central character of $\sig$ and define, for $d>0$, 
\begin{equation} \label{E:pi0}
\pi_d^\pm(\theta) = \begin{cases}
\Ind_{B\K_d}^\K (\sig_0^\pm)^{\eta^d} & \textrm{if $\theta$ coincides with the central character of $\sig_0$;}\\
\Ind_{B\K_d}^\K \tau (\sig_0^\pm)^{\eta^d} & \textrm{otherwise.}
\end{cases}
\end{equation}
Then for any $\sigma=\sigma(\om)$ with central character $\theta$ we have
$$
\Ind_{B\K_d}^\K \sig^{\eta^d} \cong \pi_d^+(\theta) \oplus \pi_d^-(\theta)
$$
where each $\pi_d^\pm(\theta)$ has degree $\frac12(q^2-1)q^{d-1}$.  We now offer an alternative description of these representations as per Theorem~\ref{T:Shalika}, and as a consequence deduce their irreducibility.

%In the next proposition we deduce the irreducibility of each $\pi_d^\pm(\om)$ from its  equivalence with an irreducible representation of the form $\Sh_d(\rho,X)$ from Theorem~\ref{T:Shalika}.

\begin{proposition} \label{P:zeropm}
Let $\theta$ be a character of $Z$.  Write $X(a,0)$ as usual for $\smat{0&a\\0&0}$.  Then for each $d>0$, we have
$\pi_d^+(\theta) \cong \Sh_d(\theta,X(-\p^{-d},0))$ and $\pi_d^-(\theta) \cong \Sh_d(\theta,X(-\ep\p^{-d},0))$.
Consequently, each $\K$-representation $\pi_d^\pm(\theta)$ is irreducible.
\end{proposition}

\begin{proof}
Fix $x \in \{1,\ep\}$ and set $X_x = X(-x\p^{-d},0)$.  Let $\zeta$ be the nontrivial character of $\resk^\times$ of order $2$, inflated to a character of $\R^\times$.  Its kernel is $(\R^\times)^2$.   For any $s\in \R^\times$, let $\sgn(s)\in \{+,-\}$ denote the sign of $\zeta(s) \in \{\pm 1\}$.  In these terms, we need to show that
\begin{equation}\label{E:tworeps}
\pi_d^{\sgn(x)}(\theta) \cong \Sh_d(\theta,X_x).
\end{equation}
We henceforth write $X$ for $X_x$.  We first show that $\Sh_d(\theta,X)$ is well-defined.  The centralizer of $X$ in $\K$ is $T(X)=ZU$ where  $U=\GG_\alpha(\R)$ is the unipotent upper triangular subgroup.  We extend $\theta$ trivially over $U$ to a character, also denoted $\theta$, of $T(X)$.  We easily verify that the characters $\Psi_X$ and $\theta$ are both trivial on the intersection $T(X)\cap \Gmess$, so we may apply Theorem~\ref{T:Shalika} to conclude that $\Sh_d(\theta,X)$ is well-defined and an irreducible representation of $\K$.  

The representations in \eqref{E:tworeps} have the same degree so it suffices to show that they admit nonzero intertwining.   By Frobenius reciprocity and Mackey theory, it suffices to show that
$$
\Ind_{B\K_d \cap T(X)\Gmess}^{B\K_d} \Psi_{\theta,X}
$$
intertwines on $B\K_d$ with either $(\sig_0^{\sgn(x)})^{\eta^d}$ or $\tau(\sig_0^{\sgn(x)})^{\eta^d}$, according to central character.  Since these representations have (maximal) depth $d$, they factor through the finite group quotient $B\K_d/\K_{d+1}$.  Thus our approach is to evaluate the characters of these representations of finite groups and calculate their intertwining number $\I$.

First, we need some additional notation.  
Let $S$ be a set of representatives of $(\resk^\times)^2$ in $\R^\times$.  
For any $u\in \R^\times$ (or, since $\Psi$ factors to a character of the quotient $\resk$,  in $\resk^\times$) define
$$
\xi_{u} = \sum_{y\in S} \Psi(uy).
$$
Since $\Psi$ is trivial on $\PP$ $\xi_u$ takes on only one of two values, $\xi_1$ and $\xi_\ep$.   We have $\overline{\xi_u} = \xi_{-u}$ and we compute directly that
$
\sum_{u\in \resk^\times} \xi_u \xi_{-u} = (q^2-1)/4.
$

Let $\ch$ denote the central character of $\sig_0^\pm$.   
Referring again to \cite{DigneMichel1991}, the character of $\sig_0^{\sgn(x)}$ is nonzero only on classes of elements the form $\smat{z&0\\c&z}$, where it is given by $\frac12(q-1)\theta_0(z)$ if $c=0$ and $\xi_{-zcx}\theta_0(z)$ if $c\neq 0$.  
Therefore the character $\chi_d^{x}$ of $(\sig_0^{\sgn(x)})^{\eta^d}$ on an element $g=(g_{ij}) \in B\K_d$ is given
by
\begin{align*}
\chi_d^{x}(g) &= \Tr\left((\sig_0^{\sgn(x)})^{\eta^d} \left( \mat{g_{11} & g_{12} \\ g_{21} & g_{22}}\right)\right) \quad \textrm{with $g_{11},g_{22}\in \R^\times, g_{12}\in \R$ and $g_{21}\in \PP^d$}\\
&= \Tr \left(\sig_0^{\sgn(x)} \left( \mat{g_{11} & g_{12}\p^d \\ g_{21}\p^{-d} & g_{22}} \right) \right)\\
&= \begin{cases}
\frac12(q-1)\ch(z) & \text{if $g_{11} \in z + \PP$ for some $z =\pm 1$, and $g_{21}\in \PP^{d+1}$;}\\
\xi_{-xzg_{21}\p^{-d}}\ch(z)  & \text{if $g_{11} \in z + \PP$ for some $z=\pm 1$, and $g_{21} \in \PP^{d}\setminus \PP^{d+1}$,}\\
0 & \textrm{otherwise.}
\end{cases}
\end{align*}
The character of $\tau(\sig_0^{\sgn(x)})^{\eta^d}$ is given by $\tau\chi_d^{x}$.

On the other hand, to calculate the character $\psi_d^{x}$ of $\Ind_{B\K_d \cap T(X)\Gmess}^{B\K_d} \Psi_{\theta,X}$ we use the Frobenius formula.  Since $B\K_d \cap T(X)\Gmess$ is normal in $B\K_d$, $\psi_d^{x}(g)=0$ for $g\notin B\K_d \cap T(X)\Gmess$.   For the remaining values of $g$, we need a set of coset representatives $\Delta$ of $B\K_d/(B\K_d \cap T(X)\Gmess)$; we choose
$$
\Delta = \left\{ \left. \delta_a = \mat{a&0\\0&a^{-1}} \right| a\in \R^\times / \pm \U_{\lrc{d/2}}  \right\}.% + \PP^{\lrc{d/2}}) \right\}.
$$      
For each $g \in B\K_d  \cap T(X)\Gmess$ choose a factorization $g=t u$ with $t= \smat{z&v\\0&z} \in T(X)$ and $u = (u_{ij}) \in \Gmess\cap B\K_d$.  Then $z\in \{\pm 1\}$ satisfies $z \equiv g_{11}$ mod $\PP^{\lrc{d/2}}$.  We compute $\Psi_{\theta,X}(tu)=\theta(t)\Psi(\Tr(Xu)) = \theta(z)\Psi(-x\pi^{-d}u_{21})$.   Since $u_{21} \equiv zg_{21}$ mod $\PP^{\lrc{d/2}}$, this simplifies to $\Psi_{\theta,X}(g) =\theta(z)\Psi(-xzg_{21}\p^{-d})$.
%For any such choice, we have
%$$
%z \equiv g_{11} \mod \PP^{\lrc{d/2}}, \quad \textrm{and} \quad
%u_{21} \equiv zg_{21} \mod \PP^{\lrc{d/2}},
%$$
%$\Psi_{\theta,X}(g) =\theta(z)\Psi(-xzg_{21}\p^{-d})$.
%Given an element of $B\K_t  \cap T(X)\Gmess[t/2]$ we may write it as $\tau u=\tau(a,b)(u_{ij})$ with $\tau(a,b)\in T(X)$ and $u = (u_{ij}) \in \Gmess[t/2]$; in particular, $u_{21} = c\p^{t}$ for some $c\in \R$.  Note that for any $\delta \in \Delta$, $\delta^{-1}\tau u\delta = \tau(\delta^{-1}u\delta)$, so by direct calculation we have $\Psi_{X_{x\p^{-t}}}(\delta^{-1}u\delta) = \Psi(xcd^2)$.   
Hence
\begin{align*}
\psi_d^{x}(g) 
&= \sum_{\delta_a \in \Delta}\Psi_{\theta,X}(\delta_a^{-1}g\delta_a)\\
&=\sum_{a \in \R^\times / \pm \U_{\lrc{d/2}}}\theta(z)\Psi(-xzg_{21}\p^{-d}a^2), \\
&= \begin{cases}
 \theta(z) \frac{q-1}{2}q^{\lrc{d/2}-1} & \textrm{if $g_{21} \in \PP^{d+1}$;}\\
\theta(z)\xi_{-xzg_{21}\p^{-d}} q^{\lrc{d/2}-1} & \textrm{if $g_{21} \in \PP^d\setminus \PP^{d+1}$,}\\
0 &\textrm{if $g\notin B\K_d  \cap T(X)\Gmess$.}
\end{cases}
\end{align*}
Thus
%Treating these representations as representations of the finite groups $B\K_t/\K_{t+1}$, their intertwining number is
\begin{align*}
I(\chi_d^{x},\psi_d^{x}) &= \frac{1}{\vert B\K_d/\K_{d+1} \vert} \sum_{g\in B\K_d/\K_{d+1}} \chi_d^{x}(g) \overline{\psi_d^{x}(g)}\\
&= \frac{1}{(q-1)q^{2d+2}} \sum_{g\in (B\K_d \cap T(X)\Gmess)/\K_{d+1}} \chi_d^{x}(g) \overline{\psi_d^x(g)}\\
&= \frac{1}{(q-1)q^{2d+2}} \left( \sum_{\substack{z\in \pm1\\ g_{11},g_{12}}} \ch(z)\theta(z) \frac{(q-1)^2}{4}q^{\lrc{d/2}-1}\right.\\
&\quad +
 \left. \sum_{\substack{z\in \pm1,g_{11},g_{12}\\ c\in \resk^\times}} 
\ch(z)\theta(z)\xi_{-xzc}\xi_{xzc}q^{\lrc{d/2}-1} \right)
\end{align*}
where the sums are over all $g_{11} \in z + \PP^{\lrc{d/2}}$ and $g_{12} \in \R$, taken modulo $\PP^{d+1}$,  and where to simplify the expression we have set $c=g_{21}\p^{-d}$ in the second sum, so with $g_{21}$ taken modulo $\PP^{d+1}$ this is effectively a sum over $\resk^\times$.  

We see that the first term vanishes if $\ch(-1)\neq \theta(-1)$; otherwise it has sum 
$$
2\left(\frac14 (q-1)^2q^{\lrc{d/2}-1}\right) \vert \PP^{\lrc{d/2}}/\PP^{d+1}\vert \; \vert \R/\PP^{d+1}\vert = \frac12 (q-1)^2 q^{2d+1}.
$$
Similarly, the second sum vanishes unless $\ch(-1)=\theta(-1)$, in which case it gives
$$
q^{2d+1} \sum_{c\in \resk^\times} \left(\xi_{-xc}\xi_{xc} + \xi_{xc} \xi_{-xc} \right) = 
q^{2d+1}\left( 2\sum_{u\in \resk^\times}\xi_{u}\xi_{-u}\right) = \frac{q^2-1}{2}q^{2d+1}.
$$
Thus these representations intertwine only when $\ch=\theta$, in which case
$$
I(\chi_d^{x},\psi_d^x)  = \frac{1}{(q-1)q^{2d+2}} \left(\frac12 (q-1)^2 q^{2d+1} +\frac{q^2-1}{2}q^{2d+1}\right)=1.
$$
On the other hand, when $\theta = \tau \ch$ we have instead that $I(\tau\chi_d^{x},\psi_d^x)=1$.
%Putting this all together, one sees that when $\ch(-1)=\om_0(-1)$, then $\I(\chi_{t,+},\ch^x) = 1$ when $-x \in  (\R^\times)^2$ and
%$\I(\chi_{t,-},\ch^x) = 1$ when $-x \notin  (\R^\times)^2$; and that this intertwining number is zero in all other cases.
%The theorem follows.
\end{proof}

We summarize the conclusions of this section in the following theorem.

\begin{theorem}[Branching Rules for Depth-Zero Supercuspidal Representations] \label{T:depthzero}
Let $\sig = \sig(\om)$ with $\om^2 \neq 1$, and let $\theta$ denote its central character.  Then the decomposition into irreducible $\K$-representations of the restrictions to $\K$ of the corresponding depth-zero supercuspidal representations are given by
$$
\Res_{\K}\cind_{\K}^\G \sig \cong \sig \oplus \bigoplus_{t \geq 1} \left( \pir_{2t}^+(\cht) \oplus \pir_{2t}^-(\cht)\right)
$$
and
$$
\Res_{\K}\cind_{\K^\eta}^\G \sig^\eta \cong \bigoplus_{t \geq 1} \left( \pir_{2t-1}^+(\cht) \oplus \pir_{2t-1}^-(\cht)\right).
$$
On the other hand, for $\sig = \sig_0^\pm$, which each have central character $\ch$, the decompositions are given by
$$
\Res_{\K}\cind_{\K}^\G \sig_0^\pm \cong \sig_0^\pm \oplus \bigoplus_{t \geq 1} \pir_{2t}^\pm(\ch)
$$
and
$$
\Res_{\K}\cind_{\K^\eta}^\G (\sig_0^\pm)^\eta \cong \bigoplus_{t \geq 1} \pir_{2t-1}^\pm(\ch).
$$
\end{theorem}

\section{Positive depth case} \label{S:posdepth}
We reprise our notation for positive depth supercuspidal representations. %, as in Section~\ref{S:supercuspidals}.

\begin{theorem} \label{T:61}
Let $(\T,y,r,\phi)$ be a generic tamely ramified cuspidal $\G$-datum with $r>0$.  Let $\rho = \rho(\T,y,r,\phi)$ and set $s=r/2$.  Choose $\Aphi \in \LieT_{-r}$ representing the restriction of $\phi$ to $\T_{s+}$ as in \eqref{E:Aphi}.
Let $\mu = \alp^t\lambda\in \Mu(\T)$ and set $d = r+\delta(\mu)$. Then
%and set $d=r+y+2t$ unless $\lambda = \w$, when $d=r-y+2t$.   Then whenever $d>r$ we have
\begin{equation} \label{E:two}
\Ind_{\K \cap (\T\G_{y,s})^{\mu}}^\K \rho^{\mu} \cong \Sh_d(\phi^{\mu},\Aphi^{\mu}),
\end{equation}
unless $y=0$ and $t=0$, in which case $\Ind_{\T\G_{0,s}}^{\K}\rho$ is not equivalent to any $\Sh_r(\phi,\Aphi)$. 
%where $d=r+y+2t$ unless $\lambda = \w$, when $d=r-y+2t$.  , \eqref{E:two} also holds for $t=0$, $\lambda=1$ and $y>0$, so that $d>r$.
\end{theorem}

%Note that the case which is excluded in the theorem, that is, satisfying $d=r$, is that of $y=0$ and $t=0$, where the left hand side of \eqref{E:two} is already as simplified as possible.

\begin{proof}
The final statement of the theorem is evident by comparison of degrees and depth; in fact $\Ind_{\T\G_{0,s}}^\K \rho$ is explicitly an example of what Shalika classified as an (irreducible) unramified representation in \cite{Shalika1967}.

So assume $t>0$ or $y\neq 0$.  By Proposition~\ref{P:split} we have that
\begin{equation} \label{E:firstrep}
\Ind_{\K \cap (\T\G_{y,s})^{\mu}}^\K \rho^{\mu} =\Ind_{(\K \cap \T^\mu)(\K \cap \G_{y,s}^{\mu})}^\K \rho^{\mu}  
\end{equation}
whereas by definition,
\begin{equation} \label{E:secondrep}
\Sh_d(\phi^\mu,\Aphi^\mu) = \Ind_{T(\Aphi^\mu)\Gmess}^\K \Psi_{\phi^\mu,\Aphi^\mu}.
\end{equation}
These two representations have the same depth and degree, by Theorem~\ref{T:Shalika} and Propositions~\ref{P:depth} and \ref{P:degree}.
%,  and Theorem~\ref{T:Shalika}, respectively.  They have the same degree $\frac12(q^2-1)q^{d-1}$ by Proposition~\ref{P:degree} and Theorem~\ref{T:Shalika}, respectively.  Both representations are induced, so by Mackey theory it suffices to prove that 
Since $\Sh_d(\phi^\mu,\Aphi^\mu)$ is irreducible, it suffices to prove that they intertwine; in particular, it suffices to show that 
$\rho^\mu$ and $\Psi_{\phi^\mu,\Aphi^\mu}$ intertwine on the intersection
$$
\left( (\K \cap \T^\mu)(\K \cap \G_{y,s}^{\mu}) \right) 
\cap \left( T(\Aphi^\mu)\Gmess \right).
$$
Noting that $T(\Aphi^\mu) = \K \cap \T^\mu$, this intersection is simply
\begin{equation} \label{E:intgroup}
T(\Aphi^\mu)(\G_{y,s}^{\mu}\cap\Gmess).
\end{equation}
When $\G_{y,s} = \G_{y,s+}$, $\rho$ is a character, and the restriction of $\rho^\mu$ to an element $tu$ of this group, with $t\in T(\Aphi^\mu)$ and $u \in \G_{y,s}^{\mu}\cap\Gmess$, is given by 
$
\rho^\mu(tu) = \phi^\mu(t) \Psi_\Aphi^\mu(u) = \phi^\mu(t) \Psi_{\Aphi^\mu}(u) = \Psi_{\phi^\mu,\Aphi^\mu}(tu),
$ 
so 
\eqref{E:firstrep} and \eqref{E:secondrep} clearly intertwine.

%, and on this subgroup, the character $\Psi_{\phi^\mu,\Aphi^\mu}$ is given by $\phi^\mu$.
%Comparing the inducing subgroups, we have that $T(\Aphi^x)=\K \cap T(\Aphi)^x = \K\cap \T^x$ whereas $(\K \cap \G_{y,s}^{x}) \cap \G_{[0,\frac12],\frac12 d}$ is a strictly smaller subgroup.
%$$
%\left\{ \mat{a&b\\c&d} \mid a-1,b, d-1 \in \PP^{\lceil d/2\rceil}, c \in \PP^{\lceil s+y\rceil + 2t} \right\}.
%$$
%On $T(\Aphi^x)$, the character $\Psi_{\phi^x,\Aphi^x}$ is given by $\phi^x$; on $\G_{[0,\frac12],\frac12 d}$ it is given by $\Psi_{\Aphi^x} = (\Psi_\Aphi)^x$.

%In the case that $\G_{y,s} = \G_{y,s+}$, this coincides exactly with the restriction of $\rho^x$ to the intersection of these inducing subgroups, and so by Mackey theory
%$$
%\Hom_{\K \cap (\T\G_{y,s})^x \cap T(\Aphi^x)\G_{[0,\frac12],\frac12 d}}(\rho^x,\Psi_{\phi^x,\Aphi^x}) \neq \{0\},
%$$
%whence the result.

When $\G_{y,s} \neq \G_{y,s+}$, that is, when $s$ is an integer and $\T$ is unramified so $y\in \{0,1\}$, the inducing representation $\rho^\mu$ has degree $q$.  In what follows we require the additional hypothesis that $\delta(\mu)>0$,%d>r$, 
which excludes exactly the case $t=0$, $y=0$.

By Lemma~\ref{L:Tdeltamu}, we have $T(\Aphi^\mu) = Z\T_{\delta(\mu)}^\mu \subseteq Z\T_{0+}^\mu$, so by Part (1) of Lemma~\ref{L:Yu}, we know the restriction of $\rho^\mu$ to $T(\Aphi^\mu)$ is $\phi^\mu$-isotypic.  However, as
$$
\G_{y,s}^\mu \cap \G_{[0,\frac12],d/2} = 
\left\{ \mat{\U_{\lrc{d/2}} & \PP^{\lrc{d/2}} \\
\PP^{s+\delta(\mu)} & \U_{\lrc{d/2}} } \right\}\cap \G \not\subseteq \G_{y,s+}^\mu,
$$
Part (2) of the lemma does not apply.  In fact, although $\Aphi$ is uniquely determined by $\phi$ in $\LieT_{-r}$ only modulo $\LieT_{-s}$, we can see that the restriction of $\Psi_{\Aphi^\mu}$ to the above subgroup depends on the choice of $\Aphi$ modulo $\LieT_{-s+1}$.   We claim that for any of these $q$ choices of $\Aphi$ modulo $\LieT_{-s+1}$, $\Psi_{\phi^\mu,\Aphi^\mu}$ intertwines with $\rho^\mu$, and hence that the isomorphism \eqref{E:two} is independent of the choice.

To see this, note that the restriction of $\rho^\mu$ factors through the quotient 
$$
H = (\G_{y,s}^\mu \cap \G_{[0,\frac12],d/2})/ (\G_{y,r+}^\mu \cap \G_{[0,\frac12],d/2}).
$$
The group $\G_{y,r+}^\mu \cap \G_{[0,\frac12],d/2}$ is given as follows.
Set $A = \max\{r+1, \lceil d/2\rceil\}$, $B = \max\{ r+1-\delta(\mu),\lceil d/2\rceil\}$ and $C= r+1+\delta(\mu)$.  Then
$$
\G_{y,r+}^\mu \cap \G_{[0,\frac12],d/2} = \left\{ \mat{\U_{A} & \PP^B \\ \PP^{C} & \U_{A}} \right\} \cap \G.
$$
Since $2(\lceil y/2\rceil+s+t)>r$ and $(\lceil y/2\rceil+s+t)+(y+s+2t) > r+y+2t$, the quotient group $H$ is abelian.  Thus the restriction 
of $\rho^\mu$ to $H$ decomposes as a sum of $q$ characters.  The distinct characters of $H$ are given by $\Psi_Y$, where $Y$ is chosen from the dual lattice quotient; more precisely, the distinct characters correspond to elements $Y$ of the set
$$
\widehat{H} = \left\{ \mat{f & g \\ h & -f} \mid f \in \PP^{-A}/\PP^{-\lceil d/2 \rceil}, h \in  \PP^{-B}/\PP^{-\lceil d/2 \rceil}, g \in \PP^{-C}/\PP^{-s-\delta(\mu)} \right\}.
$$
We thus have $\Res_{\G_{y,s}^\mu \cap \G_{[0,\frac12],d/2}}\rho^\mu = \oplus_i \Psi_{Y_i}$ for some $Y_i \in \widehat{H}$.   By Part (2) of Lemma~\ref{L:Yu}, $\rho^\mu$ is $\Psi_{\Aphi^\mu}$-isotypic on $\G_{y,s+}^\mu$.  Thus the characters $\Psi_{Y_i}$ and $\Psi_{\Aphi^\mu}$ must coincide on their restriction to $\G_{y,s+}^\mu$.  The elements of $\widehat{H}$ satisfying this condition are exactly those of the form
$$
Y(x) = x\mat{0&\p^{-C}\\0&0} + \Aphi^\mu,
$$
for some $x\in \R$.   We conclude that $\Res_{\G_{y,s}^\mu \cap \G_{[0,\frac12],d/2}}\rho^\mu = \oplus_{x\in \resk}\Psi_{Y(x)}$.  In particular, since $Y(0)=\Aphi^\mu$, it follows that the representation $\rho^\mu$ and the character $\Psi_{\phi^\mu,\Aphi^\mu}$ intertwine on the intersection \eqref{E:intgroup}, whence the result follows by Mackey theory. % we ma so by Mackey theory, the two representations intertwine.
\end{proof}

%\begin{remark}
%Note that $\Aphi$ is uniquely defined by $\phi$ only modulo $\LieT_{-s}$, whereas $\Aphi^\mu$ depends on the choice of representative of $\Aphi$ in $\LieT_{-r}$.  That the equivalence class of $\Sh_d(\phi^\mu,\Aphi^\mu)$ is indeed independent of these choices follows from this theorem.
%\end{remark}

We summarize the result in the following theorem.

\begin{theorem}[Branching Rules for Positive-Depth Supercuspidal Representations] \label{T:positivedepth}
Let $\rho = \rho(\T,y,r,\phi)$ for a generic tamely ramified cuspidal $\G$-datum of positive depth $r$, and let $\Aphi$ be a $\G$-generic element of depth $r$ representing $\phi$. 
Then the decomposition into irreducible $\K$-representations of the restriction to $\K$ of the corresponding supercuspidal representation of $\G$  is, for $y=0$, $\frac12$ and $1$, respectively:
\begin{align*}
\Res_{\K}\cind_{\T\G_{0,s}}^\G \rho &\cong \Ind_{\T\G_{0,s}}^\K \rho \oplus 
\bigoplus_{t > 0} \left(\Sh_{r+2t}(\phi^{\alp^t},\Aphi^{\alp^t}) \oplus  \Sh_{r+2t}(\phi^{\alp^t\ecrep},\Aphi^{\alp^t\ecrep}) \right),
\\
\Res_{\K}\cind_{\T\G_{\frac12,s}}^\G \rho &\cong \Sh_{r+\frac12}(\phi,\Aphi) \oplus %\Ind_{\T\G_{\frac12,s})}^\K \rho \oplus 
\bigoplus_{t > 0} \left(\Sh_{r+\frac12+2t}(\phi^{\alp^t},\Aphi^{\alp^t}) \oplus  \Sh_{r-\frac12+2t}(\phi^{\alp^t\w},\Aphi^{\alp^t\w}) \right),
\\
\Res_{\K}\cind_{\T\G_{1,s}}^\G \rho &\cong  \Sh_{r+1}(\phi,\Aphi) \oplus %\Ind_{\T_{1,\ep\p^2}\G_{1,s})}^\K \rho \oplus 
\bigoplus_{t > 0} \left(\Sh_{r+2t+1}(\phi^{\alp^t},\Aphi^{\alp^t}) \oplus  \Sh_{r+2t+1}(\phi^{\alp^t\ecrep^\eta},\Aphi^{\alp^t\ecrep^\eta}) \right).
\end{align*}
\end{theorem}

\section{Intertwining results} \label{S:intertwining}

%We now proceed to a discussion of the intertwining over $\K$ of different supercuspidal representations of $\G$, culminating in Corollary~\ref{C:intertwining} and Theorem~\ref{T:intertwining}.

In this section, we answer the question of when, and to what extent, two supercuspidal representations of $\G$ intertwine as representations of $\K$.  By the results in the preceding sections, this can be reduced to determining the equivalences between Shalika's representations.  

\subsection{Equivalences among Shalika's ramified representations}

Using the notation of Section~\ref{SS:repR}, the following theorem is deduced from \cite{Shalika1967}.

\begin{theorem}[Shalika] \label{T:Shalika2}
Suppose $X_1=X(u_1,v_1)$ and $X_2=X(u_2,v_2)$.  For each $i\in \{1,2\}$ suppose $-d=\val(u_i)<\val(v_i)$ and let $\theta_i$ be a character of $T(X_i)$ which agrees with $\Psi_{X_i}$ on $T(X_i)\cap \Gmess$.    
Then
\begin{enumerate}
\item If $X_1=X_2$ then 
$\Sh_d(\theta_1,X_1) \cong \Sh_d(\theta_2,X_2)$ if and only if $\theta_1 = \theta_2$.
\item If $\Sh_d(\theta_1,X_1) \cong \Sh_d(\theta_2,X_2)$ then there exists a diagonal matrix $g \in \K$ such that $\Psi_{X_1} = \Psi_{X_2}^g$. 
\end{enumerate}
\end{theorem}

Recall that $\Psi_{X_1} =\Psi_{X_2}$ as characters of $\Gmess$ only if $u_1\equiv u_2$ modulo $\PP^{\lrc{-d/2}}$ and $v_1\equiv v_2$ modulo $\PP^{\lrc{(-d+1)/2}}$.

\begin{proof}[Comments on proof]
The first statement comes from \cite[Thm 4.2.1, Thm 4.2.5]{Shalika1967}.  It follows from Clifford theory, since $T(X)\Gmess$ is the normalizer of $\Psi_X$ in $\K$ and thus $\Ind_{\Gmess}^{T(X)\Gmess}\Psi_X$ decomposes as a multiplicity-free direct sum of characters of the form $\Psi_{\theta,X}$. 

If $d$ is odd, then the second property is stated explicitly in \cite[Thm 4.2.1]{Shalika1967} as $\Psi_{X_1} = \Psi_{X_2}^g = \Psi_{X_2^g}$ for some $g\in \K$; by our choice of representatives of the orbits we may without loss of generality assume $g$ is diagonal.  

If $d$ is even, then the second property is instead implicit in the proof of \cite[Thm 4.2.5]{Shalika1967}.  Namely, the statement of that theorem gives a $g\in\K$ conjugating $X_1$ to some $X_1'$ where $X_1'\equiv X_2$ modulo $\g_{0,-d/2}$.  Thus $g$ may be assumed to be diagonal and replacing $\Sh_d(\theta_1,X_1)$ with its conjugate by $g$, we may assume that $X_1=X_1'$.  Then $\Psi_{X_1}$ and $\Psi_{X_2}$ are equal upon restriction to $\G_{0,(d+1)/2}\subsetneq \Gmess$.  One checks that the normalizer of $\Res_{\G_{0,(d+1)/2}}\Psi_{X_i}$ is $T(X_i)\G_{0,d/2}$, whence $T(X_1)\G_{0,d/2}= T(X_2)\G_{0,d/2}$.

By \cite[Lemma 4.1.1]{Shalika1967}, we know that if a character $\psi$ of a subgroup $H$ occurs in the restriction to $H$ of $\Xi_1$ and $\Xi_2$, where each $\Xi_i$ is an irreducible representation of the normalizer $N(\psi)$ of $\psi$ in $\K$, then $\Ind_{N(\psi)}^\K \Xi_1 \cong \Ind_{N(\psi)}^\K \Xi_2$ if and only if $\Xi_1 \cong \Xi_2$.  
In our case, for each $i\in \{1,2\}$ set
$$
\Xi_i = \Ind_{T(X_i)\Gmess}^{T(X_i)\G_{0,d/2}}\Psi_{\theta_i,X_i};
$$ 
clearly $\Res_{\G_{0,(d+1)/2}}\Psi_{X_i}$ occurs in the restriction of each $\Xi_i$.  Since $\Ind_{T(X_1)\G_{0,d/2}}^\K \Xi_1 = \Sh_d(\theta_1,X_1) \cong \Ind_{T(X_1)\G_{0,d/2}}^\K \Xi_2$
is irreducible, so are the $\Xi_i$, so applying the above result yields $\Xi_1 \cong \Xi_2$.

% and we have $T(X_1)\G_{0,d/2}= T(X_2)\G_{0,d/2}$.  Now $\Psi_{X_1}$ occurs in the restriction to $\G_{0,(d+1)/2}$ of $\Ind_{T(X_1)\Gmess}^{T(X_1)\G_{0,d/2}}\Psi_{\theta_1,X_1}$, and $T(X_1)G_{0,d/2}$ is the normalizer of $\Res_{G_{0,(d+1)/2}}\Psi_{X_1}$.  

%Since $\Ind_{T(X_1)\G_{0,d/2}}^\K \Xi_1 = \Sh_d(\theta_1,X_1) \cong \Ind_{T(X_1)\G_{0,d/2}}^\K \Xi_2$ is irreducible, we can conclude (by \cite[Lemma 4.1.1]{Shalika1967}, for example) that $\Xi_1 \cong \Xi_2$.
 
As noted at the beginning of the proof, $\Xi_i$ occurs as an irreducible component of $\Ind_{\Gmess}^{T(X_i)\G_{0,d/2}}\Psi_{X_i}$, and so these representations of $T(X_1)\G_{0,d/2}=T(X_2)\G_{0,d/2}$ intertwine.  But $\Gmess$ is normal in $T(X_1)\G_{0,d/2}$ so by  \cite[Lemma 4.1.2]{Shalika1967} the representations are isomorphic, and there exists some $h\in \T(X_1)\G_{0,d/2}$ such that 
 $\Psi_{X_1} = \Psi_{X_2}^h$.    Because both $X_1$ and $X_2$ are already antidiagonal, this forces $\Psi_{X_1} = \Psi_{X_2}$, as required.
\end{proof}

Throughout this section, given $X=X(u,v)$ we denote by ``$t(a,b)\in T(X)$'' the element $\smat{a&b\\bu^{-1}v & a}$, leaving the dependence on $u$ and $v$ implicit.  Recall that for any $u\in \ratk^\times$, if $X=X(u,0)$ then $T(X)=ZU$ for $U$ the unipotent upper triangular subgroup of $\K$.

\begin{corollary} \label{C:triangle}
Suppose $X=X(u,0)$ and $X'=X(u',0)$, with $\val(u)=\val(u')=-d$.  Then $\Sh_d(\theta,X)\cong \Sh_d(\theta',X')$ if and only if there is some $c\in \R^\times$ such that $c^2u=u'$ and for all $t(z,b)\in T(X)$, we have $\theta(t(z,c^{-2}b))=\theta'(t(z,b))$.
\end{corollary}

\begin{proof}
By part (2) of Theorem~\ref{T:Shalika2}, if $\Sh_d(\theta,X)\cong \Sh_d(\theta',X')$ then there exists some $g=\diag(c,c^{-1})\in \K$ such that $\Psi_X^g = \Psi_{X'}$.    Thus $c^2u \equiv u'$ modulo $\PP^{\lrc{-d/2}}$, so replacing $c$ by a suitable multiple if necessary we may assume $c^2u=u'$.  Therefore $\Sh_d(\theta,X)\cong \Sh_d(\theta^g,X^g)\cong\Sh_d(\theta',X')$ and $X^g=X'$, so by part (1), we must have $\theta^g = \theta'$, which is the property sought.  The converse is immediate.
\end{proof}

We wish to generalize this Corollary to provide an explicit complement to part (2) of Theorem~\ref{T:Shalika2}.  The challenge lies in that in general $T(X_1)\neq T(X_2)$ even though $\Psi_{X_1}=\Psi_{X_2}$ and thus $T(X_1)\Gmess=T(X_2)\Gmess$.

Let $X=X(u,v)$ and $X' = X(u',v')$ be given and set $n= \val(u^{-1}v-u'^{-1}v')$.  Suppose $t'= t(a',b) \in T(X')$; so $a'^2 = 1+b^2u'^{-1}v'$.  Let $a \in \R$ be the unique solution to 
$
a^2 = 1 + b^2u^{-1}v
$
satisfying $a\equiv a' \mod \PP^n$.  Let $t$ be the corresponding element $t(a,b)\in T(X)$, which we call the \emph{transfer} of $t'$ to $T(X)$.  The transfer is not a homomorphism but has the property that
\begin{equation}\label{E:tinvt}
t^{-1}t' = \left[ \begin{matrix} 
1+(a-a')a' & (a-a')b\\ 
(a-a')bu^{-1}v+ab(u'^{-1}v'-u^{-1}v) & 1+a(a-a') \end{matrix} \right] \in \G_{0,n}.
\end{equation}
In particular, when  $T(X')\Gmess = T(X)\Gmess$ (equivalently, when $n\geq \lrc{(d+1)/2}$, which implies $t^{-1}t'\in \Gmess$) the transfer gives a refactorization 
\begin{equation}\label{E:refactorization}
T(X')\Gmess \ni t'g' = t ((t^{-1}t')g') = tg \in T(X)\Gmess.
\end{equation}

\begin{proposition} \label{P:theta}
Let $X=X(u,v)$ with $\val(v) > \val(u) = -d$, and suppose $\theta$ is a character of $T(X)$ extending $\Psi_X$.  Suppose $X' = X(u',v')$ is such that $\Psi_X = \Psi_{X'}$.  For any $t'\in T(X')$ let $t\in T(X)$ be its transfer and set
\begin{equation} \label{E:theta}
\theta'(t') = \theta(t)\Psi_X(t^{-1}t').
\end{equation}
Then $\theta'$ is a character of $T(X')$ extending $\Psi_{X'}$ and $\Sh_d(\theta,X) \cong \Sh_d(\theta',X')$.  Moreover, if $u^{-1}v\equiv u'^{-1}v'$ modulo $\PP^{d+1}$ then $\Sh_d(\theta,X) \cong \Sh_d(\theta',X')$ if and only if 
$\theta(t)=\theta'(t')$ for all such transfer pairs $(t,t')$.
\end{proposition}

\begin{proof}
Since  $\Psi_X = \Psi_{X'}$, these characters have the same normalizer $T(X)\Gmess = T(X')\Gmess$ and so $n = \val(u^{-1}v - u'^{-1}v') \geq \lrc{(d+1)/2}$.  Since the induced representation
$\Ind_{\Gmess}^{T(X)\Gmess}\Psi_X$ decomposes as a multiplicity-free direct sum of characters of the form $\Psi_{\theta,X}$, we conclude that for each character $\theta$ of $T(X)$ extending $\Psi_X$ there is a unique character $\theta'$ of $T(X')$ extending $\Psi_{X'}$ such that $\Psi_{\theta,X}=\Psi_{\theta',X'}$.  It suffices to verify that $\theta'$ satisfies the given identity.

Given $t'=t(a',b) \in T(X')$ and its transfer $t = t(a,b) \in T(X)$, we may for any $g'\in \Gmess$ refactorize $t'g' = t(t^{-1}t'g')$ as in \eqref{E:refactorization}.  Evaluating $\Psi_{\theta,X}=\Psi_{\theta',X'}$ on both sides yields
$$
\theta'(t')\Psi_{X'}(g')=\theta(t)\Psi_X(t^{-1}t')\Psi_X(g')
$$
whence \eqref{E:theta}. %the result, since $\Psi_{X'}=\Psi_{X}$.
For the final statement, we use the explicit form \eqref{E:tinvt} to compute
\begin{equation} \label{E:psix}
\Psi_X(t^{-1}t') = \Psi(abu(u'^{-1}v'-u^{-1}v))\Psi(2bv(a-a')).
\end{equation}
Therefore this correction factor 
disappears if and only if $n>d$.  Note that in this case the transfer realizes an isomorphism $T(X)/\ker(\theta)\cong T(X')/\ker(\theta')$.  
\end{proof}

\begin{remark}
By applying this proposition to a character $\phi$ of $\T$ and two choices of good elements of depth $-r$ representing $\phi$, say $\Aphi$ and $\Aphi'=f\Aphi$, for some $f\in \U_{\lrc{s}}$, one obtains explicitly the equivalence $\Sh_d(\phi^\mu,\Aphi^\mu) \cong \Sh_d(\phi^\mu,\Aphi'^\mu)$ for any $\mu$; this was implicitly a consequence of Theorem~\ref{T:61}.
\end{remark}

\subsection{Equivalences among $\K$-components of supercuspidal representations of positive depth $r$: case $d>2r$}

From now on, let $\pi$ be a supercuspidal representation of $\G$ of positive depth arising from a (generic tamely ramified cuspidal $\G$-) datum $(\T, y,r,\phi)$.  As usual set $s = r/2$.  Let $\mu = \alp^t\lambda \in \Mu(\T)$ with $t>0$.   Then $\mu$ parametrizes an irreducible representation of $\K$ of depth $d = r+\delta(\mu)$ occuring in $\Res_{\K}\pi$ which we denote $\pi_\mu$.  We have shown that if $\Aphi = \aphi X_\T \in \LieT_{-r}$ represents $\Res_{\T_{s+}}\phi$, then $\pi_\mu \cong \Sh_d(\phi^\mu,\Aphi^\mu)$.  Let  $X(u,v)=\Aphi^\mu$; then $\val(u)=-d$ and $\val(v) = -d+2\delta(\mu) = \delta(\mu)-r=d-2r$.

\begin{remark} \label{Remark:simple}
A key property we exploit is the following.  Let $m>s$ and $m \geq \delta(\mu)$.  If $t \in \T_m^\mu = \K_{m-\delta(\mu)}\cap T(\Aphi^\mu)$, then 
$$
\phi^\mu(t) = \phi(t^{\mu^{-1}}) = \Psi_\Aphi(t^{\mu^{-1}})
$$
which is computed simply as $\Psi(\Tr(\Aphi(t^{\mu^{-1}}-I)))=\Psi(\Tr(\Aphi^\mu(t-I)))$.  That is, $\Aphi^\mu$ determines the value of the Shalika inducing character on $\T_m^\mu \Gmess \supseteq \Gmess$.
\end{remark}

%We begin with a special case.

\begin{proposition} \label{P:triangle}
Suppose $d>\frac43 r$.  Given  $\Aphi^\mu = X(u,v)$, set $X' = X(u,0) \in \g_{0,-d}$.  Then there exists a character $\theta'$ of $ZU=T(X')$, which agrees with $\phi$ on $Z$, such that
$$
\Sh_d(\phi^\mu,\Aphi^\mu) \cong \Sh_d(\theta',X').
$$
Explicitly, for each $t' = t(z,b)\in ZU$, let $t\in T(\Aphi^\mu)$ be its transfer; then 
\begin{equation} \label{E:tprime}
\theta'(t')= \phi^\mu(t) \Psi(zbv)^{-1}.
\end{equation}
If $d\leq 2r$ then the depth of $\theta'$ is $r-\delta(\mu)=2r-d$.  Otherwise, $\theta'$ is trivial on $U$.
\end{proposition}

\begin{proof}
Note that $d>\frac43 r$ implies $\val(v)\geq \lrc{(-d+1)/2}$, so $\Psi_{\Aphi^\mu}=\Psi_{X'}$ on $\Gmess$. The existence of $\theta'$ then follows from Proposition~\ref{P:theta}.  Let $t' = t(z,b)\in ZU$ have transfer $t = t(a,b)\in T(\Aphi^\mu)$.  Then $a\equiv z$ modulo $\PP^{2d-2r}$, which for $d>\frac43 r$ implies  $(a-z)bv \in \PP$.  Therefore \eqref{E:psix} yields
$
\Psi_{\Aphi^\mu}(t^{-1}t') = \Psi(-abv)= \Psi(zbv)^{-1}
$
and so \eqref{E:theta} simplifies to \eqref{E:tprime} in this case.

Evidently $\theta'$ and $\phi$ coincide on $Z$.
To determine the depth of $\theta'$, take $z=1$ and $b\in \PP^m$, for  $m=\max\{0,r-\delta(\mu)\}$.  Then the transfer of $t(1,b)\in ZU$ to $t=t(a,b)\in T(\Aphi^\mu)$ lies in $\K_m\cap T(\Aphi^\mu)$.  Lemma~\ref{L:Tdeltamu} implies $t\in \T_{\delta(\mu)+m}^\mu \subseteq \T_r^\mu \subseteq \T_{s+}^\mu$.  Therefore by Remark~\ref{Remark:simple}, $\phi^\mu(t)=\Psi(\Tr(\Aphi^\mu(t-I))) = \Psi(2bv)$.
We thus simply require $\theta'(t(1,b))=\Psi(bv)$.  Since $\val(v)=\delta(\mu)-r$, this is trivial for $b\in \PP^{m+1}$ and, if $m>0$, is nontrivial for $b\in \PP^m$.  When $d>2r$, we have $\val(bv)>0$ for all $b\in \R$ and so $\theta'(t(1,b))=1$.
\end{proof}

\begin{corollary} \label{C:intertwining}
Suppose that $d > 2r$.  Then $\pi_\mu$ depends only on the central character $\theta$  of $\phi$ and the class of $u\p^d$ modulo ${\R^\times}^2$, where $\Aphi^\mu = X(u,v)$.  In particular, with notation as in \eqref{E:pi0}, we have
$
\pi_\mu \in \{ \pi_d^+(\theta), \pi_d^-(\theta) \}.
$
\end{corollary}

\begin{proof}
Indeed, when $d>2r$, Proposition~\ref{P:triangle} implies that $\pi_\mu \cong \Sh_d(\theta',X(u,0))$, where in effect $\theta'$ is the trivial extension of the central character of $\pi$ to $ZU$.
Corollary~\ref{C:triangle} implies that this representation depends on $u$ only up to an element of ${\R^\times}^2$, whence either $\pi_\mu \cong \Sh_d(\theta,X(-\p^{-d},0)) \cong \pi_d^+(\theta)$ or  $\pi_\mu \cong \Sh_d(\theta,X(-\ep\p^{-d},0)) \cong \pi_d^-(\theta)$.
% where $\theta'(t') = \phi^\mu(t)\Psi(t^{-1}t)$ for $t\in T(\Aphi^\mu)=Z\T_{\delta(\mu)}^\mu$ the transfer of $t'=t(z,b)\in ZU$.  Since $\val(u^{-1}v)=2d-2r>d$, we have $\Psi(t^{-1}t)=1$; since $\delta(\mu)=d-r>r$, $\phi^\mu(t)=\phi(z)$.  Hence $\theta'(t(z,b))=\phi(z)$ so $\theta'=\theta$, the extension of the central character of $\pi$ to $ZU$.  From Corollary~\ref{C:triangle} we deduce that this representation depends on $u$ only up to its square class, implying that either $\pi_\mu \cong \Sh_d(\theta,X(-\p^{-d},0)) = \pi_d^+(\theta)$ or  $\pi_\mu \cong \Sh_d(\theta,X(-\ep\p^{-d},0)) = \pi_d^-(\theta)$, following Proposition~\ref{P:zeropm}.
\end{proof}

Thus for any supercuspidal representation $\pi$ of $\G$, its $\K$-components of sufficiently large depth $d>2r$ occur among the four irreducible representations of $\K$ --- corresponding to classes in $\R^\times/(\R^\times)^2$ and characters of $Z$  --- which occur as depth-$d$ components of the decompositions of 
supercuspidal representations of depth zero.  Further implications of this result, also vis-\`a-vis principal series representations, are explored in \cite{Nevins2011}.

\subsection{Equivalences among $\K$-components of supercuspidal representations of positive depth $r$: case $r<d\leq 2r$}

%Our final question concerns the possibility of intertwining among $\K$-components of depth $d$ between distinct supercuspidal representations, in the case that $r<d\leq 2r$.  We begin with some immediate cases.

We begin with some immediate examples.

\begin{proposition} \label{P:1}
Let $(\T,y,r,\phi)$ be a datum giving rise to a supercuspidal representation of positive depth.  Let $\psi$ be a nontrivial character of $\T$ of depth $m < r$ which is trivial on $Z$.  Set $\phi' = \psi \phi$ and construct the supercuspidal representation $\pi'$ corresponding to $(\T, y, r, \phi')$.  Then $\pi_\mu \cong \pi'_\mu$ if and only if $\delta(\mu)>m$.
\end{proposition}

\begin{proof}
By construction $\phi'$ also has depth $r$, so there exists $\Aphi' \in \LieT_{-r}$, a good element of depth $-r$ representing $\phi'$.  We have $\pi_\mu \cong \Sh_d(\phi^\mu,\Aphi^\mu)$ and $\pi'_\mu \cong \Sh_d({\phi'}^\mu,{\Aphi'}^\mu)$.

If $m\leq s$, then $\phi$ and $\phi'$ agree on $\T_{s+}$, so we may choose $\Aphi'=\Aphi$.  Applying part (1) of Theorem~\ref{T:Shalika2} yields $\pi_\mu\cong \pi'_{\mu}$ if and only if $\phi^\mu = \phi'^\mu$, which happens if and only if $\psi^\mu=1$ as a character of $T(\Aphi^\mu)=Z\T_{\delta(\mu)}^\mu$.  Since $\psi$ is trivial on $Z$, this happens if and only if $\delta(\mu)>m$, as required.

Otherwise,  we have $s<m<r$.   
%Write $\Aphi' = \Aphi+\Aphi_\psi$ for some $\Aphi_\psi \in \LieT_{-m}$.
%Since $\psi = \phi'\phi^{-1}$ has depth $m$, it is represented on $\T_{\frac{m}{2}+}\supseteq \T_{s+}$ by some $\Aphi_\psi \in \LieT_{-m}$.  Consequently, we may choose $\Aphi' = \Aphi+\Aphi_\psi$.  
%Since $\Aphi, \Aphi' \in \LieT$, there is some $f\in \ratk^\times$ such that $\Aphi = f\Aphi'$; 
Write $\Aphi'=f \Aphi$ for some $f\in \R^\times$; since $\phi$ and $\phi'$ agree precisely on $\T\cap \G_{y,m+}$, $f-1\in\PP^{r-m}\setminus \PP^{r-m+1}$ and so $f \in (\R^\times)^2$.  Choose $c\in \R^\times$ such that $c^2=f^{-1}$ and set $g=\diag(c,c^{-1})\in \K$.
Writing $\Aphi^\mu = X(u,v)$ we have $(\Aphi'^\mu)^g = X(fu,fv)^g = X(u,f^{2}v)$.  

Now $\val(f^{2}v-v) = m-\delta(\mu)$; if this is less than $\lrc{(-d+1)/2}$ then $\Psi_{\Aphi^\mu}$ and $\Psi_{(\Aphi'^\mu)^g}$ are not conjugate by any diagonal matrix and we conclude by part (2) of Theorem~\ref{T:Shalika2} that $\pi_\mu \not\cong \pi'_{\mu}$.   
%Note that  $\val((f^{2}-1)v)= (r-m) + (\delta(\mu)-r)= \delta(\mu)-m$, so this was the case $\delta(\mu)< m+  \lrc{(-d+1)/2}$.
%If $\delta(\mu) \geq m+  \lrc{(-d+1)/2}$, 
Otherwise, we simply have $\Psi_{\Aphi^\mu} = \Psi_{(\Aphi'^\mu)^g}$.  We apply 
Proposition~\ref{P:theta} with $X=(\Aphi'^\mu)^g$, $\theta=(\phi'^\mu)^g$ and $X'=\Aphi^\mu$ to deduce that ${\pi'_\mu}^g \cong \Sh_d(\theta',\Aphi^\mu)$,  where 
$$
\theta'(t') = (\phi'^\mu)^g(t)\Psi_{\Aphi^\mu}({t}^{-1}t')
$$
for each $t' \in T(\Aphi^\mu)$ with transfer $t \in T((\Aphi'^\mu)^g)$.
Since $\pi'_\mu \cong {\pi'_\mu}^g$, it follows from part (1) of Theorem~\ref{T:Shalika2} that $\pi_\mu \cong \pi'_\mu$ if and only if $\phi^\mu=\theta'$ as characters of $T(\Aphi^\mu)=Z\T_{\delta(\mu)}^\mu$.

%.  For any $t' \in T((\Aphi'^\mu)^g)$, let $t \in T(\Aphi^\mu)$ be its transfer.  Then by Proposition~\ref{P:theta}, $\Sh_d(\phi^\mu,\Aphi^\mu) \cong \Sh_d((\phi'^\mu)^g,(\Aphi'^\mu)^g)$ --- and hence $\pi_\mu \cong \pi'_\mu$ --- if and only if the identity $ (\psi^\mu\phi^\mu)^g(t') = \phi^\mu(t)\Psi_{\Aphi^\mu}(t^{-1}t')$ holds.

Let $n = \max\{\delta(\mu),m\}$.  Let $t'=t(a',b) \in T_n^\mu$, with $a \equiv 1$ mod $\PP$, and let $t=t(a,b) \in T((\Aphi'^\mu)^g)$ be its transfer.  Then
$t^{g^{-1}}=t(a,bf) \in T(\Aphi'^\mu)$ also lies in $\T_{n}^\mu$.  Since $m>s$, $\phi$ and $\phi'$ are given by $\Psi_\Aphi$ and $\Psi_{\Aphi'}$, respectively, on $\T_n$.  Using Remark~\ref{Remark:simple}
we find that $\phi^\mu(t') = \Psi(2bv)$ and $(\phi'^\mu)^g(t) = \phi'^\mu(t^{g^{-1}})=\Psi(2bf^2v)$.  Also, noting that $2m>r$ implies under these circumstances that $abv\equiv a'bv \equiv bv$ modulo $\PP$, we compute using \eqref{E:psix} that 
$\Psi_{\Aphi^\mu}({t}^{-1}t') = \Psi((1-f^2)bv)$.  Thus $\theta'(t) = \Psi((1+f^2)bv)$.  Note that $(1+f^2)bv-2bv\in \PP^{n-m}\setminus \PP^{n-m+1}$.  If $\delta(\mu)<m$, then $n=m$ and so $\theta' \neq \phi^\mu$ on $\T_n^\mu$, whence $\pi_\mu \not\cong \pi'_\mu$.  If  $\delta(\mu)\geq m$, then $n=\delta(\mu)$ and we conclude that $\theta'$ and $\phi^\mu$ coincide as characters of $\T_{\delta(\mu)}^\mu$.  Since they also agree on $Z$ by construction, we have $\pi_\mu \cong \pi'_\mu$.
\end{proof}

\begin{proposition} \label{P:2}
Given $(\pi,\T,y,r,\phi)$ as above, let $\psi$ be a character of $\T$ of depth $m<r$ which is trivial on $Z$.  Let $\phi' = \psi \phi^{-1}$ and construct the supercuspidal representation $\pi'$ corresponding to $(\T, y, r, \phi')$.  Then if $-1\in (\ratk^\times)^2$, we have $\pi_\mu \cong \pi'_\mu$ whenever $\delta(\mu)>m$, whereas if $-1\notin (\ratk^\times)^2$, we have $\pi_\mu \cong \pi'_{\mu'}$ whenever $\delta(\mu)=\delta(\mu')>m$ and $\mu\neq \mu'$.
\end{proposition}

\begin{proof}
By Proposition~\ref{P:1}, it suffices to prove the result for $\psi=1$, where $m=0$.  Since now $\phi' = \phi^{-1}$, we may choose $\Aphi' = -\Aphi$.  

If $-1\in (\ratk^\times)^2$, then let $g=\diag(c,c^{-1})\in \K$ where $c^2=-1$.  For such $g$, and any $\mu$, we have $(\Aphi^\mu)^g = -\Aphi^\mu = \Aphi'^\mu$.  For any $t(a,b)\in T(\Aphi^\mu)$, we have $t(a,b)^{g^{-1}} = t(a,c^{-2}b) = t(a,-b) = t(a,b)^{-1}$, so for all $t\in T(\Aphi^\mu)=T(\Aphi'^\mu)$, we have $(\phi^\mu)^g(t)=\phi^\mu(t^{-1})=\phi'^\mu(t)$.  We thus conclude
$$
\pi_\mu \cong \Sh_d(\phi^\mu,\Aphi^\mu) \cong \Sh_d((\phi^\mu)^g,(\Aphi^\mu)^g) = \Sh_d(\phi'^\mu,\Aphi'^\mu) \cong \pi'_{\mu}.
$$
Now suppose $-1 \notin (\ratk^\times)^2$, $\delta(\mu)=\delta(\mu')>0$ and $\mu\neq \mu'$.  The latter two conditions can be satisfied only if $\T$ is unramified.  We have assumed in this case that $\ep=-1$ (otherwise one must conjugate appropriately) so from $\Aphi=-\Aphi'$ we obtain $\Aphi^\mu = \Aphi'^{\mu'}$.  We verify directly that for $t\in T(\Aphi^\mu)$, $t^{(\mu'\mu^{-1})}=t^{-1}$, so that $\phi'^{\mu'}(t) = \phi^\mu(t)$.  The equivalence now follows as above.
\end{proof}

Finally, let us show that the cases arising in Propositions~\ref{P:1} and \ref{P:2} are in fact exhaustive.

%, namely that $\T=\T'$ and $\phi$ differs from either $\phi'$ or $\phi'^{-1}$ by a character of depth less than $r$ which is trivial on $Z$, are the only cases in which two distinct supercuspidal representations share a common component of depth less than or equal to $2r$.

\begin{theorem}\label{T:intertwining}
Suppose $(\T,y,r,\phi)$ and $(\T',y',r',\phi')$ are data defining two supercuspidal representations of positive depth, denoted $\pi$ and $\pi'$ respectively.  Suppose they contain a common $\K$-component of depth $d$, with $r<d\leq 2r$.  Then
\begin{itemize}
\item $r=r'$, that is, $\pi$ and $\pi'$ have the same depth;
\item $y=y'$ and $\T=\T'$, or more generally, their defining tori are conjugate; and
\item $\phi$ and $\phi'$ are related as in one of Propositions~\ref{P:1} or \ref{P:2}.
\end{itemize}
\end{theorem}

\begin{proof}
Suppose the common irreducible $\K$-component is $\pi_\mu \cong \pi'_{\mu'}$.  Then $\Sh_d(\phi^\mu,\Aphi^\mu) \cong \Sh_{d'}(\phi'^{\mu'},\Aphi'^{\mu'})$ and %$d=r+\delta(\mu) = r'+\delta(\mu')=d'$.  
so $d=d'$.  Set $X(u,v)=\Aphi^\mu$ and $X(u',v')=\Aphi'^{\mu'}$.  Then $\val(u) = \val(u')= -d$ whereas \emph{a priori} $\val(v) = d-2r$ and $\val(v')=d-2r'$.
We first show that $\val(v)=\val(v')$.

  %Then $\val(v) < \frac{-d+1}{2}$ and  $-r <\lrc{\frac{-d+1}{2}}-\delta(\mu) \leq \lrc{-s}$.  
Since  $\Sh_d(\phi^\mu,\Aphi^\mu) \cong \Sh_{d}(\phi'^{\mu'},\Aphi'^{\mu'})$, Theorem~\ref{T:Shalika2} implies there exists  $c \in \R^\times$ such that (scaling if necessary) we have
\begin{equation} \label{E:c2u}
c^2u = u' \quad \textrm{and} \quad c^{-2}v \equiv v' \mod \PP^{\lrc{(-d+1)/2}}.
\end{equation}
If $r<d \leq \frac43 r$, then $\val(v)=d-2r\leq -\frac12 d <\lrc{(-d+1)/2}$.  Thus $v$ is nonzero modulo $\PP^{\lrc{(-d+1)/2}}$ and the second congruence implies $\val(v') =\val(v)$.

Otherwise, we have $\frac43 r< d \leq 2r$. %, so that $r-\delta(\mu) = 2r-d\geq 0$. 
By the preceding we also have $d> \frac43 r'$; without loss of generality we may assume $r \geq r'$.    Then Proposition~\ref{P:triangle} applies, yielding characters $\theta_1$ and $\theta_2$ of $ZU$ such that 
$$
\pi_\mu \cong \Sh_d(\theta_1,X_1) \quad \textrm{and} \quad \pi'_{\mu'} \cong \Sh_d(\theta_2,X_2) 
$$
where $X_1 = X(u,0)$ and $X_2=X(u',0)$.  Since $c^2u=u'$, Corollary~\ref{C:triangle} implies that for all $t=t(z,b)\in ZU$, 
$$
\theta_1(t(z,c^{-2}b))=\theta_2(t(z,b)).
$$  
Let $t(a,c^{-2}b)\in T(\Aphi^\mu)$ 
be the transfer of $t=t(z,c^{-2}b)\in T(X_1)$ and let $t(a',b)\in T(\Aphi'^{\mu'})$ be the transfer of $t(z,b) \in T(X_2)$.  Expanding the above equality using \eqref{E:tprime} yields
\begin{equation}\label{E:step}
\phi^\mu(t(a,c^{-2}b))\Psi(zc^{-2}bv)^{-1} = \phi'^{\mu'}(t'(a',b))\Psi(zbv')^{-1}.
\end{equation}
This already implies $\phi$ and $\phi'$ agree on $Z$, so suppose $z=1$ and let $\val(b) = r-\delta(\mu)=2r-d\geq 0$.  Then $t(a,c^{-2}b)\in \T_{r}^\mu$, so by Remark~\ref{Remark:simple} we compute
$$
\phi^\mu(t(a,c^{-2}b)) = \Psi_{\Aphi}(t(a,c^{-2}b)^{\mu^{-1}}) %= \Psi(2ac^{-2}bv) 
= \Psi(2bc^{-2}v).
$$
A similar argument, and our hypothesis $r\geq r'$, gives $\phi'^{\mu'}(t(a',b)) = \Psi(2bv')$.
Therefore for $z=1$ and $b \in \PP^{r-\delta(\mu)}\setminus\PP^{r-\delta(\mu)+1}$ \eqref{E:step} becomes the identity
$$
\Psi(bc^{-2}v)= \Psi(bv'),
$$
and by construction the left side is not identically $1$.  We conclude  that $c^{-2}v\equiv v'$ modulo $\PP^{\val(v)+1}$, whence in particular $\val(v)=\val(v')$.

Thus in both cases we have that $\val(v)=\val(v')$, whence $\delta(\mu)=\delta(\mu')$ and $r=r'$.  From the definition of $\delta$ we then conclude $y \equiv y'$ modulo $2\Z$, whence $y=y'$.  

To prove the final statement, begin by noting that in both cases above we have found $c\in \R^\times$ such that $c^2u=u'$ and $c^{-2}v\equiv v'$ modulo $\PP^{\val(v)+1}$.  

If $\mu \neq \mu'$ then $\delta(\mu)=\delta(\mu')$ implies $y \in \{0,1\}$.  
Thus  $\T=\T'$ is an unramified torus 
and we may assume without loss of generality that $\mu = \alp^t$ and $\mu' =\alp^t\dcrep$ where $\Lambda(\T)=\{1,\dcrep\}$.  Since $\Aphi$ and $\Aphi'$ are good elements of depth $-r$, we can write $\Aphi = \aphi X_\T$ and $\Aphi' = \aphi'X_\T$ with in this case $\val(\aphi)=\val(\aphi')=-r$.  Using Table~\ref{Table:xt}, we can write
%\begin{equation}\label{E:Aphimumu}
$$
\Aphi^\mu = X(\aphi \p^{-\delta(\mu)}, \aphi \ep \p^{\delta(\mu)}) \quad
\textrm{and} \quad \Aphi'^{\mu'} = X(\aphi' \ep \p^{-\delta(\mu)}, \aphi' \p^{\delta(\mu)}).
$$ %\end{equation}
Thus $c^2= u'/u = \aphi'\ep/\aphi$.  On the other hand, the congruence $c^{2} \equiv v/v'$ modulo $\PP$ yields
$c^2 \equiv \aphi\ep/\aphi'$ modulo $\PP$.  Thus modulo $\PP$, the quotient $\aphi/\aphi'$ is a self-invertible nonsquare, which exists if and only if $-1 \notin (\ratk^\times)^2$.  So taking $\ep=-1$, we have simply $\Aphi \equiv -\Aphi'$ modulo $\LieT_{-r+1}$.  The character $\psi = \phi \phi'$ is represented on $\T_{r}$ by $\Aphi+\Aphi'$ and so is trivial there, hence is of depth $m<r$.  It is also trivial on $Z$; we are thus in the setting of Proposition~\ref{P:2} and we deduce that $m<\delta(\mu)$.  

So now we may assume that $\mu = \mu'$.

If $y\in \{0,1\}$, then $\T=\T'$ and an argument as above yields  $c^2=\aphi'/\aphi$ and $c^2 \equiv \aphi/\aphi'$ modulo $\PP$.  We deduce that $\Aphi \equiv z\Aphi'$ modulo $\LieT_{-r+1}$ where this time $z$ is a self-invertible square.  A similar argument applies, showing that if $z=1$ then $\phi'$ can be factored as $\psi \phi$ whereas if $z=-1\in (\ratk^\times)^2$, then $\phi'$ can be factored as $\psi \phi^{-1}$, where in each case the depth of $\psi$ is less than $\delta(\mu)$.  Thus we are in the case of Proposition~\ref{P:1} or \ref{P:2}, respectively.

If $y=\frac12$, then there are up to $4$ possible choices for $\T=\T_{\gam_1,\gam_2}$ and $\T'=\T_{\gam_1',\gam_2'}$.  We suppose that $\mu=\alp^t$;  the case that $\mu=\alp^t\w$ (so $\delta(\mu)=2t-y$) is accomplished by interchanging 
$\gam_1$ with $\gam_2$ (and $\gam'_1$ with $\gam'_2$) throughout the following argument.  Again by Table~\ref{Table:xt}, for some $\aphi, \aphi'$ of valuation $-r-y$, we have
$$
\Aphi^{\alp^t} = X(\aphi \gam_1 \p^{-2t}, \aphi \gam_2 \p^{2t}) \quad 
\textrm{and} \quad
{\Aphi'}^{\alp^t} = X(\aphi' \gam'_1 \p^{-2t}, \aphi' \gam'_2 \p^{2t}).
$$
Thus $c^2=\frac{\aphi'\gam_1'}{\aphi \gam_1}$ and $c^{-2} \equiv  \frac{\aphi\gam_2}{\aphi'\gam_2'}$ modulo $\PP$.
It follows that the quotients $\gam_2'/\gam_2$ and $\gam_1'/\gam_1$ lie in the same class modulo $(\ratk^\times)^2$.  

We claim it cannot be the case that both are nonsquare.  For if they were then Table~\ref{Table:tori} would imply: that $\T\neq \T'$;  that $\T$ and $\T'$ correspond to the same ramified extension field; and that the quotients are mutually inverse.   Repeating the argument of the case $\mu\neq \mu'$ above, we would conclude that $-1 \notin (\ratk^\times)^2$, in which circumstance the tori are conjugate, hence equal by our choices, which is a contradiction.  

So both quotients are squares, whence from  Table~\ref{Table:tori} we see $\T=\T'$ and we have $\gam_i=\gam'_i$ for $i\in \{1,2\}$.  Then as before there exists a $z\in \{\pm 1\} \cap (\ratk^\times)^2$ such that 
$\Aphi \equiv z\Aphi'$ modulo $\LieT_{-r+1}$, and we are done, as above.
\end{proof}

%However, given that any other representative $\Aphi' \in \LieT_{-r}$ is equal to $f\Aphi$ for some $f\in \U_\lrc{s}$, we can choose $c$ such that $c^2=f^{-1}$ so that $C = \diag(c,c^{-1})$ yields $\Psi_{\Aphi'^C} = \Psi_{\Aphi}$.  One can then verify explicitly via Lemma~\ref{L:theta} that $\Sh_d((\phi^\mu)^C,\Aphi'^C)\cong \Sh_d(\phi^\mu,\Aphi)$, whence  $\Sh_d(\phi^\mu,\Aphi')\cong \Sh_d(\phi^\mu,\Aphi)$ as expected.

\subsection*{Acknowledgments}  My grateful thanks to Jeff Adler, Fiona Murnaghan and Loren Spice for patiently sharing their knowledge of Bruhat-Tits buildings and supercuspidal representations with me.

\bibliography{padicrefs}{}
\bibliographystyle{amsplain}

\end{document}